\documentclass[10pt]{amsart}
\usepackage{version,amsmath, amssymb}
\usepackage{color}
\usepackage{amsthm, amsopn, amsfonts, xypic}
\usepackage[colorlinks=true]{hyperref}
\hypersetup{urlcolor=blue, citecolor=blue}
\usepackage[color=yellow]{todonotes}

\let\arXiv\arxiv

\def\doi#1{ {\href{http://dx.doi.org/#1}
   {{\mdseries\ttfamily DOI}}}}

\let\MR\mr

\newcommand{\tH}{\tilde H}

\newcommand{\newsection}[1]
{\section{#1}\setcounter{theorem}{0} \setcounter{equation}{0}
\par\noindent}

\newtheorem{theorem}{Theorem}

\newtheorem{lemma}[theorem]{Lemma}

\newtheorem{proposition}[theorem]{Proposition}
\newtheorem{definition}[theorem]{Definition}

\newcommand{\cd}{\, \cdot\, }

\newcommand{\R}{{\mathbb R}}

\newcommand{\ang}{{\not\negmedspace\nabla}}

\renewcommand{\L}{{\mathcal{L}}}

\newcommand{\la}{\langle}
\newcommand{\ra}{\rangle}
\newcommand{\LE}{\mathcal{LE}}

\renewcommand{\S}{{\mathbb S}}

\newcommand{\tg}{{\tilde{g}}}

\newcommand{\Max}{\text{Max}}

\begin{document}

\title { Pointwise decay for the Maxwell field on black hole
  space-times }

\author{Jason Metcalfe} \address{Department of Mathematics, University
  of North Carolina, Chapel Hill, NC 27599-3250}
\email{metcalfe@email.unc.edu}

 \author{Daniel Tataru} \address{Department of Mathematics, University of California,
  Berkeley, CA 94720} \email{tataru@math.berkeley.edu} 

\author{Mihai  Tohaneanu} \address{Department of Mathematics, University of Kentucky, Lexington, KY 40506} \email{mihai.tohaneanu@uky.edu}

\thanks{The first author was supported in part by NSF grant
  DMS-1054289, the second author by DMS-1266182 as well as by a Simons Investigator award from the
Simons Foundation, and the third author by DMS-1636435.}

\begin{abstract}
  In this article we study the pointwise decay properties of solutions
  to the Maxwell system on a class of nonstationary asymptotically
  flat backgrounds in three space dimensions. Under the assumption
  that uniform energy bounds and a weak form of local energy decay
  hold forward in time, we establish peeling estimates, as well as a $t^{-4}$ rate of decay on compact regions for all the
  components of the Maxwell tensor. 
\end{abstract}

\maketitle

\newsection{Introduction} In this article we consider the question of
pointwise decay for solutions to the Maxwell system with localized
initial data.  The class of backgrounds we are interested in are
certain asymptotically flat black hole backgrounds, e.g of
Schwarzschild/Kerr type and perturbations thereof. However, the type
of results we obtain in this article treat a compact set essentially
as a black box, so they also apply in other settings.  Our interest in
this problem originates from general relativity, where the Maxwell (or
spin 1) system is a linearized model of the Einstein Equations that
captures some of the difficulties not present in the scalar wave
equation (or spin 0) case.

The main idea of this article is that the pointwise decay bounds are a
consequence of local energy decay estimates for the same Maxwell
system, even though the local energy decay bounds are invariant with
respect to time translations, while the pointwise decay bounds are
not. This fits into the philosophy that  the local energy decay estimates 
are the core decay estimates, and the other types of decay estimates 
(e.g. Strichartz, pointwise) are derived bounds. In the context of the 
Schr\"odinger equation on asymptotically flat space-times, this 
approach was developed in \cite{Tat}, \cite{MMT}. More recently,
the same philosophy was implemented in the context of the scalar 
wave  equation, beginning with \cite{MT}. The case of the scalar 
wave equation on black hole space times is discussed in what follows.

We begin with local energy estimates for solutions to the scalar wave
equation $\Box_g u = f$ on Schwarzschild and Kerr manifolds, which
have been recently established by various authors (\cite{BS1},
\cite{BS2}, \cite{BSter}, \cite{DaRoLE}, \cite{DaRo09}, \cite{MMTT} for
Schwarzschild, \cite{TT}, \cite{DaRoNotes}, \cite{AB} for Kerr with
small angular momentum, and \cite{DaRoNew}, \cite{DaRoNew2}, \cite{DaRoNew3} for
Kerr with $|a| < M$).  The transition from local energy decay to
Strichartz estimates was considered in \cite{MMTT}, \cite{To}.  The
key result that sharp decay bounds
(Price's Law \cite{Price1}) follow from the local
energy decay was first obtained in \cite{Tat} for stationary
space-times, using time Fourier transform and resolvent analysis, and
then in the nonstationary case in \cite{MTT}, by using more robust
methods based on the classical vector field method.
 (See also \cite{DSS}, \cite{DSS2} for a more refined  Fourier based analysis applied to
 Schwarzschild space-times.)

 The main result in the present article is the exact counterpart of
 \cite{MTT} in the context of the Maxwell system, and asserts that
 local energy decay implies sharp\footnote{At least for $r \geq \frac{t}2$; understanding what happens in the interior of a small cone 
 seems to be a more delicate matter.}  pointwise decay bounds. These can be
 seen as Price's law in the Maxwell setting; indeed, \cite{Price2} conjectures a decay rate of $t^{-5}$ in compact regions for the Maxwell system on the Schwarzschild metric.  
 
 Since our result is a conditional one, it is useful to review where we stand as far as
 local energy decay estimates are concerned.  With regards to the
 Maxwell system on Schwarzschild, a class of local energy estimates (as well as
 some partial pointwise rates of decay) were established in \cite{Blue} for solutions to the homogeneous system with no
 charge. For solutions to the homogeneous system on Kerr spacetimes
 with small angular momentum $|a|\ll M$ there is recent work
 \cite{AB2} that establishes some local energy estimates
 and uniform energy bounds.
 
 For the inhomogeneous system with charges, the article \cite{StTat}
 provides local energy estimates in a variety of spherically
 symmetric spacetimes, including Schwarzschild.  This is the context
 where the results in the present paper directly apply.  We expect the
 analogous estimates for the Kerr spacetimes and small perturbations
 thereof to also hold, in which case the same decay results would be
 true. We remark that the results in  \cite{AB2} cannot be used directly
in our present context, as they only deal with solutions to the homogeneous 
equation; however, if one had the appropriate inhomogeneous version  of the 
bounds in \cite{AB2}, that would suffice.  

 While there are substantial similarities between our present result
 for the Maxwell system and our earlier work \cite{MTT} for the scalar
 wave equation, there are also some significant differences. Some of
 these differences are of a technical nature and stem from the fact
 that we are dealing with a first order hyperbolic system rather than
 with a first order scalar wave equation. 

 However, there is also a significant conceptual difference, which is
 that even in the simplest stationary problem one has nontrivial zero
 modes to deal with.  These zero mode components are parametrized by
 the electric, respectively the magnetic charge, of the system, which
 are conserved quantities. In spherical symmetry the problem
 simplifies considerably in that these modes correspond exactly to
 the radial part of the Maxwell tensor, and thus can be easily
 eliminated. Instead of taking this easy way out, here we develop an
 approach that relies neither on the radiality nor on the
 stationarity of the metric.

A natural follow-up question to ask here would be whether our approach here 
generalizes to the spin 2 case, i.e. to linearized gravity. One obvious additional
difficulty there is that there are more zero modes present, and it is not clear to us 
whether these modes can be tracked in the same way as in the present paper,
given a general (i.e. non Kerr) black hole space-time.

\subsection{Acknowledgements}
The authors are very grateful to the anonymous referee for carefully reading the manuscript
and for many useful suggestions and corrections, including the need of condition \eqref{srderiv} for our theorem to hold.

\newsection{Notation and setup}

\subsection{Notations} We use $(t=x_0, x)$ for the
coordinates in $\R^{1+3}$. We use Latin indices $i,j=1,2,3$ for
spatial summation and Greek indices $\alpha,\beta=0,1,2,3$ for
space-time summation. In $\R^3$ we also use polar coordinates $x = r
\omega$ with $\omega \in \S^2$.  By $\la r \ra$ we denote a smooth
radial function which agrees with $r$ for large $r$ and satisfies $\la
r \ra \geq 2$.  We consider a partition of $ \R^{3}$ into the dyadic
sets $A_R= \{\la r \ra \approx R\}$ for $R \geq 1$, with the obvious
change for $R=1$.

\subsection{Space-times}
We are interested in uniformly smooth asymptotically flat Lorentzian
space-times $(M,g)$ in either $M = \R^+\times \R^3$ or an exterior region of the
form $M = \R^+\times \R^3 \setminus B(0,R_0)$.  To set a proper orientation for our space-time,
we make the following assumption:

$(i)$ The level sets $t = const$ are space-like.

To describe the regularity of the coefficients of the metric, we use the following sets of vector
fields:
\[
\partial = \{ \partial_t, \partial_i\}, \qquad \Omega = \{x_i \partial_j -
x_j \partial_i\}, \qquad S = t \partial_t + x \partial_x,
\]
namely the generators of translations, rotations and scaling. We set
$Z = \{ T,\Omega,S\}$.  Then we define the classes $S^Z(r^k)$ of
functions in $\R^+ \times \R^3$ by
\[
a \in S^Z(r^k) \Longleftrightarrow |Z^j a(t, x)| \leq c_{j} \la
r\ra^{k}, \quad j \geq 0.
\]
By $S^Z_{rad}(r^k)$ we denote spherically symmetric functions in
$S^Z(r^k)$. 

This leads us to our second main assumption.

$(ii)$  $(M,g)$ is asymptotically flat.

Here, for the purpose of the present paper, we make the
following definition:

\begin{definition}
We say that  $g$ is asymptotically flat if it has the form
\[
g = m + g_{sr} + g_{lr},
\]
where $m$ stands for the Minkowski metric, $g_{lr}$ is a stationary
long range spherically symmetric component, with $S^Z_{rad}(r^{-1})$
coefficients, of the form
\[
g_{lr} = g_{lr, tt}(r) dt^2 + g_{lr, tr}(r) dt dr + g_{lr, rr}(r) dr^2
+ g_{lr, \omega \omega}(r) r^2 d \omega^2
\]
and $g_{sr}$ is a short range component of the form
\[
g_{sr} = g_{sr, tt} dt^2 + 2g_{sr, ti} dt dx_i + g_{sr, ij} dx_i dx_j
\]
with $S^Z(r^{-2})$ coefficients which also satisfy
\begin{equation}\label{srderiv}
\partial g_{sr} \in S^Z (r^{-3}).
\end{equation}
\end{definition}

This definition is set to match the setup of relativistic space-times,
e.g. Schwarzschild and Kerr. In that context, the $O(r^{-1})$ radial 
part of the metric is associated to mass, while the $O(r^{-2})$ 
nonradial terms are associated to the angular momentum. 
Having accurate decay rates for the metric perturbation at infinity
is essential in this work; indeed, these decay rates, rather that the local behavior 
of the metric, are the factor which determines the exact decay rates 
for both scalar and electromagnetic waves. We also remark that, in
contrast with our previous result for the wave equation, we require that the derivative of the short range component gains one extra order of decay. This will prove to be crucial in gaining extra decay for the radial part in Section 6, as well as obtaining the correct bounds for the curvature tensor in Section 7.

Our decay results are expressed relative to the distance to the
Minkowski null cone $ \{t = |x|\}$. This can only be done provided
that there is a null cone associated to the metric $g$ which is within
$O(1)$ of the Minkowski null cone.  However, in general the long range
component of the metric produces a logarithmic correction to the
cone. This issue can be remedied via a change of coordinates that
roughly corresponds to using Regge-Wheeler coordinates in
Schwarzschild/Kerr near spatial infinity; see \cite{Tat}.
This is related to the fact that our asymptotic flatness condition is
stable with respect to a class of changes of coordinates $\chi$ of the
form
\[
\chi = \chi_{lr} + \chi_{sr}
\]
where $\chi_{lr}$ is radial and satisfies $\nabla \chi_{lr} - I \in
S^Z(r^{-1}) $ while $\nabla \chi_{sr} \in S^Z(r^{-2})$. This class
  allows for logarithmic cone corrections.  Indeed, after a further
  conformal transformation, the metric $g$ is
  reduced to a normal form where
\begin{equation}\label{normlr}
  g_{lr} = g_{\omega}(r) r^2 d \omega^2, \qquad g_\omega \in S^Z_{rad}(r^{-1}).
\end{equation}
In this context, we can replace $\Box_g$ by an operator
of the form
\begin{equation}\label{P}
 P = \Box + Q
\end{equation}
where $\Box$ denotes the d'Alembertian in the Minkowski metric and the
perturbation $Q$ has the form
\begin{equation}
Q  = g^{\omega} \Delta_{\omega} +
\partial_\alpha g_{sr}^{\alpha \beta}\partial_\beta + V,
\quad g_{sr}^{\alpha \beta} \in S^Z(r^{-2}), \  g^\omega\in
S^Z_{rad}(r^{-3}), \ V \in S^Z(r^{-3}).
\label{can-p-sr}\end{equation} 

See \cite{Tat} and \cite{MTT} for more details.

We call these coordinates normalized coordinates.  Most of the
analysis in the paper is done in normalized coordinates and with $g$
in normalized form.

Finally, concerning the local properties of the metric we make 
either one of the following assumptions:

$(iii)_a$ {\em (regular space-time)} $M = \R^+ \times \R^3$.

$(iii)_b$  {\em (black hole space-time)}  $M = \R^+ \times \R^3 \setminus B(0,R_0)$
and the lateral boundary $\R \times \partial B(0,R_0)$ is outgoing
space-like.

One could consider also other settings, e.g. exterior space-times $M =
\R^+ \times \R^3 \setminus \partial B(0,R_0)$ with various boundary
conditions on the time-like boundary $\R \times \partial B(0,R_0)$. 

\subsection{The Maxwell system.}
 
In spacetimes as above, we consider  a Maxwell field 
$F$, which is an antisymmetric $(0, 2)$-tensor field on a Lorentzian
manifold $(M , g)$ satisfying the Maxwell equations:
\begin{equation}\label{Maxwelleqns}
  dF = G_1, \qquad d*F = G_2. 
\end{equation}
In the physical context one  disallows magnetic currents and sets $G_1= 0$. However,
mathematically it is more convenient to work in a symmetric setting and 
allow both $G_1$ and $G_2$ to be nonzero.

We will assume that the initial data $F(0)$ at time $t= 0$ is smooth and compactly
supported. The inhomogeneous terms $G_1$ and $G_2$ should satisfy the
compatibility conditions
\[
dG_1 = dG_2 = 0,
\]
as well as be supported in the forward cone $C = \{ t \geq r - R_1\}$
for some $R_1>0$.

For comparison purposes, we also state the corresponding result for the scalar wave equation,
  \begin{equation}\label{box}
     \Box_g u = f
  \end{equation}   
with initial data $u[0] = ( u(0) , \partial_t u(0))$ at time $t = 0$. This is the problem 
considered in our preceding paper \cite{MTT}, to which we will refer repeatedly 
here.

\subsection{Local energy norms} 
We now introduce our local energy norms. For a scalar function $u$ we define
\begin{equation}
  \begin{split}
    \| u\|_{LE} &= \sup_R  \| \la r\ra^{-\frac12} u\|_{L^2 (\R^+ \times A_R)},  \\
    \| u\|_{LE[t_0, t_1]} &= \sup_R  \| \la r\ra^{-\frac12} u\|_{L^2 ([t_0, t_1] \times A_R)}, \\
    \| u(t_0, \cdot)\|_{\LE} &= \sup_R \| \la r\ra^{-\frac12} u(t_0,
    \cdot)\|_{L^2 (A_R)},
  \end{split}
  \label{ledef}\end{equation}
where the last norm applies at fixed time.  
Their $H^1$ counterparts were also used in \cite{MTT} in the study of
the scalar wave equation \eqref{box}:
\begin{equation}
  \begin{split}
    \| u\|_{LE^1} &= \| \nabla u\|_{LE} + \| \la r\ra^{-1} u\|_{LE},\\
    \| u\|_{LE^1[t_0, t_1]} &= \| \nabla u\|_{LE[t_0, t_1]} + \| \la
    r\ra^{-1} u\|_{LE[t_0, t_1]}, \\
  \| u(t_0)\|_{\LE^1} &= \| \nabla_x u (t_0,\cd)\|_{\LE} + \| \la r\ra^{-1} u(t_0,\cd)\|_{\LE}.
  \end{split}
\end{equation}
 Here and in the rest of the paper we use the abbreviation $\nabla = \nabla_{t,x}$.

The corresponding dual type spaces, used for the source terms, are:
\begin{equation}
  \begin{split}
    \| f\|_{LE^*} &= \sum_R  \| \la r\ra^{\frac12} f\|_{L^2 (\R_+ \times A_R)}, \\
    \| f\|_{LE^*[t_0, t_1]} &= \sum_R  \| \la r\ra^{\frac12} f\|_{L^2 ([t_0, t_1] \times A_R)}, \\
    \| f(t_0, \cdot)\|_{\LE^*} &= \sum_R \| \la r\ra^{\frac12}
    f(t_0,\cdot)\|_{L^2 (A_R)}.
  \end{split}
  \label{lesdef}\end{equation} 
  
 We also define similar norms for higher Sobolev
regularity 
\[
\begin{split}
  \| u\|_{LE^{1,k}} &= \sum_{|\alpha| \leq k} \| \partial^\alpha u\|_{LE^1}, \\
  \| u\|_{LE^{1,k}[t_0, t_1]} &= \sum_{|\alpha| \leq k} \| \partial^\alpha u\|_{LE^1[t_0, t_1]}, \\
  \| u\|_{LE^{k}} &= \sum_{|\alpha| \leq k} \| \partial^\alpha u\|_{LE}, \\
  \| u\|_{LE^{k}[t_0, t_1]} &= \sum_{|\alpha| \leq k} \| \partial^\alpha u\|_{LE[t_0, t_1]},
\end{split}
\]
respectively 
\[
\begin{split}
  \| f\|_{LE^{*,k}} &=  \sum_{|\alpha| \leq k}  \| \partial^\alpha f\|_{LE^{*}}, \\
  \| f\|_{LE^{*,k}[t_0, t_1]} &=  \sum_{|\alpha| \leq k}  \| \partial^\alpha f\|_{LE^{*}[t_0, t_1]}.
\end{split}  
\]

For a triplet $\Lambda = (i,j,k)$ of multi-indices $i$, $j$ and $k$ we
denote $|\Lambda| = |i|+3|j|+9k$ and
\[
u^{\Lambda} = \partial^i \Omega^j S^k u, \qquad u^{\leq m} =
(u^{\Lambda})_{|\Lambda|\le m}.
\]
The choice of weights here follows \cite{MTT}, but is somewhat arbitrary. The goal 
 is to enable us to treat all differentiated terms with better decaying coefficients  perturbatively in the $\Omega u$
equation, and to also treat  all $\Omega$ terms with better decaying coefficients  perturbatively in the $S u$
equation.

We also define, for any norm $Y$,
\[
\|u^{\leq m}\|_Y = \sum_{|\Lambda|\leq m} \|u^{\Lambda}\|_{Y}.
\]

In the case of black hole space times one also needs to contend with
trapping.  Fortunately, for our purposes here one does not need to pay
too much attention to that, and it suffices to use a rough regularity analysis.

\begin{definition}
  a) We say that the scalar wave evolution \eqref{box} has the local energy
  decay property if the following estimate holds:
 \begin{equation}\label{noderivloss}
  \| u\|_{LE^{1,k}[t_0,\infty)} \leq c_k (\|\nabla u(t_0)\|_{ H^{k}} + \|f\|_{LE^{*,k}[t_0,\infty)} ), 
\qquad k \geq 0.
\end{equation}
b) We say that the scalar wave evolution \eqref{box} has the weak local energy
  decay property if the following estimate holds:
  \begin{equation}\label{derivloss}
\| u\|_{LE^{1,k}[t_0,\infty)} \leq c_k (\|\nabla u(t_0)\|_{ H^{k+1}} + \|f\|_{LE^{*,k+1}[t_0,\infty)} ), 
\qquad k \geq 0
\end{equation}
  in either $\R \times \R^3$ or in the exterior domain (black hole) case.
  \label{d:weakle}\end{definition}

The first definition applies for the nontrapping case. The second one
is for the black hole case, where we allow for a loss of one
derivative to account for trapping effects. We remark that in  the presence of
hyperbolic trapping this loss is much more than is required. Indeed,
generally hyperbolic trapping merely produces a logarithmic loss, and
that only near the trapped set.  But that is not so relevant to our
purposes here, so we content ourselves with the more relaxed bound
\eqref{derivloss}.
 
We also give the following definition (see \cite{MTT} for the
motivation):
\begin{definition}\label{Enbddef}
  We say that the problem $\Box_g u = f$ satisfies stationary local
  energy decay bounds if on any time interval $[t_0,t_1]$ and $k \geq
  0$ we have
 \begin{equation}\label{sle}
\| u\|_{LE^{1,k}[t_0,t_1]} \lesssim_k \|\nabla u(t_0)\|_{H^k} + \|\nabla u(t_1)\|_{H^k}+ 
\| f\|_{LE^{*,k}[t_0,t_1]} + \|\partial_t u\|_{LE^{0, k}[t_0,t_1]}.
\end{equation}
\end{definition}

Let us also mention that all the definitions above can be easily
extended to a vector $\vec{u}$ of functions by considering each
component separately.

For the Maxwell tensor $F$, we need to slightly modify our energy
norms. Using Cartesian coordinates, define
\begin{equation}
  \begin{split}
    \| F\|_{LE} &= \sum_{\alpha,\beta}  \|  F_{\alpha\beta}\|_{LE},  \\
    \| F(t_0)\|_{\LE} &= \sum_{\alpha,\beta} \| F(t_0,
    \cdot)_{\alpha\beta}\|_{\LE},
  \end{split}
  \label{Mledef}\end{equation} 
and the dual norms
\begin{equation}
  \begin{split}
    \| G\|_{LE^*} &= \sum_{\alpha,\beta,\gamma}  \| G_{\alpha\beta\gamma}\|_{LE^*}, \\
    \| G(t_0)\|_{\LE^*} &= \sum_{\alpha,\beta,\gamma} \| G(t_0,
    \cdot)_{\alpha\beta\gamma}\|_{\LE^*}.
  \end{split}
  \label{Mlesdef}\end{equation}

We will also need to define higher energy norms of the tensor
$F$. Geometrically it makes the most sense to commute the system
\eqref{Maxwelleqns} with Lie derivatives of vector fields, which we
will denote by $\L_X$. Given a set of vector fields $A$, a norm $Y$, a
tensor $W$ and a positive integer $l$, define
\[
\|\L_A W\|_Y = \sum_{X\in A} \|\L_X W\|_Y
\]
\[
\L_{A^l} W = \{\L_{X_1} \cdots \L_{X_l} W: X_1 \cdots X_l \in A
\}.
\]

 Keeping the analogy with the scalar case, we also define the higher norms associated to translations
\[
 \|F\|_{LE^k} = \sum_{l\leq k} \| \L_{\partial^l} F\|_{LE}
\]
 and similarly for $\LE^k$ and their duals. We remark that the role of the $LE^1$ norms in the scalar case is now played by the $LE$ norms for the Maxwell system.
 
 On $\{t= t_0 \}$ slices, we define the higher regularity norm for $k\geq 0$
\begin{equation}\label{enonslices}
  E^{k}(t_0) = \sum_{l\leq k} \| \L_{\partial^l} F(t_0)\|_{L^2}, \qquad E(t_0) = E^0(t_0)
\end{equation}  

We will now distinguish between the radial and nonradial parts of the
tensor, as they will have different rates of decay. This is
  where things are different from the scalar case, and this is caused by the 
zero modes associated to the electric and magnetic charges. Precisely, 
with an $LE^*$ type source, one can drive up the charge inside the cone, and  
thus eliminate any chance for local energy decay. One remedy for this 
would be to factor out the charges. This works well for spherically symmetric space-times,
where the charges correspond exactly to radial modes, but in general this strategy seems 
to be unfeasible.

However, the radial mode does seem to carry the bulk of the charge 
near infinity. This motivates our present strategy, where we weigh the 
radial mode differently, in a way which is consistent with the estimates we already know   
from \cite{StTat} to hold in spherically symmetric space-times. Precisely, our 
stronger weight for the radial components of the Maxwell tensor are strong enough in order 
to guarantee that the charge remains zero at infinity. It also allows us to obtain the radial 
components by integrating the source term from infinity.

 For a function
$\psi$, we will denote by $\overline{\psi}$ its zero spherical
harmonic. Motivated by the fact that in the case of the Schwarzschild metric the zero modes are $d\omega^2$ and $r^{-2} dt\wedge dr$, we define
\[
\overline{F} = \overline{F_{tr}} dt \wedge dr +
\overline{F_{\phi\theta}} d\omega^2,
\]
respectively
\[
\overline{G} = \overline{G_{t\phi\theta}} dt \wedge d\omega^2 +
\overline{G_{r\phi\theta}} dr \wedge d\omega^2.
\]

We can now define the norms that we are mostly interested in
\[
\begin{split}
 \| F\|_{LE_{\Max}} &= \| F\|_{LE} + \| \la r\ra \overline F\|_{LE}, \\
 \| F(t_0)\|_{\LE_{\Max}} &= \| F(t_0,\cd)\|_{\LE} + \| \la r\ra \overline F(t_0,\cd)\|_{\LE},
\end{split} 
\]
and for the inhomogeneous part
\[
\begin{split}
 \| G\|_{LE_{\Max}^*} &= \| G\|_{LE^*} + \| \la r\ra \overline G\|_{LE^*}, \\
 \| G(t_0)\|_{\LE_{\Max}^*} &= \| G(t_0,\cd)\|_{\LE^*} + \|\la r\ra \overline G(t_0,\cd)\|_{\LE^*}.
\end{split} 
\]
Moreover, for a given $k\geq 0$, the higher
regularity norms associated with Sobolev regularity are set to be:
\[
\begin{split}
  \| F\|_{LE_{\Max}^{k}} &= \|F\|_{LE^k} + \|\la r\ra\bar F\|_{LE^k}  \\
  \| F(t_0)\|_{\LE_{\Max}^{k}} &= \|F(t_0)\|_{\LE^k} + \|\la r\ra\bar F(t_0)\|_{\LE^k}
\end{split}
\]
respectively
\[
\begin{split}
  \| G\|_{LE_{\Max}^{*,k}} &= \|G\|_{LE^{*,k}} + \|\la r\ra\bar G\|_{LE^{*,k}}     \\
  \| G\|_{\LE_{\Max}^{*,k}} &=  \|G(t_0)\|_{\LE^{*,k}} + \|\la r\ra\bar G(t_0)\|_{\LE^{*,k}}.  
\end{split}
\]
 Finally, for $\Lambda$ a triplet as above, let
\[
F^{\Lambda} = \L_{\partial^i} \L_{\Omega^j} \L_{S^k} F,
\]
and
\[
F^{\leq m} = (F^{\Lambda})_{|\Lambda|\le m}, \qquad \|F^{\leq m}\|_Y =
\sum_{|\Lambda|\le m} \|F^{\Lambda}\|_{Y}.
\]
 
 We now define the norms
\[
\begin{split}
 \| F^{\Lambda}\|_{LE_{\Max}} &= \| F^{\Lambda}\|_{LE} + \| \la r\ra \bar F^{\Lambda}\|_{LE}, \\
 \| F^{\Lambda}(t_0)\|_{\LE_{\Max}} &= \| F^{\Lambda}(t_0,\cd)\|_{\LE} + \| \la r\ra \bar F^{\Lambda}(t_0,\cd)\|_{\LE}, \\
 \| G^{\Lambda}\|_{LE_{\Max}^*} &= \| G^{\Lambda}\|_{LE^*} + \| \la r\ra \bar G^{\Lambda}\|_{LE^*}, \\
 \| G^{\Lambda}(t_0)\|_{\LE_{\Max}^*} &= \| G^{\Lambda}(t_0,\cd)\|_{\LE^*} + \|\la r\ra \bar G^{\Lambda}(t_0,\cd)\|_{\LE^*}.
\end{split} 
\]

We will assume that the following bounds hold:

\begin{definition}
a) We say that the problem \eqref{Maxwelleqns} has the local
  energy decay property if the following estimate holds for each $k \geq 0$:
  \begin{equation}\label{Mderivno}
   \sup_{t > t_0} E^k(t) +  \| F\|_{LE_{\Max}^{k}} \lesssim_k E^{k}(t_0) + \sum_{i=1}^2 \|G_i\|_{LE_{\Max}^{*,k}}.
  \end{equation}
b)   We say that the problem \eqref{Maxwelleqns} has the weak local
  energy decay property if the following estimate holds  for each $k \geq 0$:
  \begin{equation}\label{Mderivloss}
    \sup_{t > t_0} E^k(t) + \| F\|_{LE_{\Max}^{k}} \lesssim_k E^{k+1}(t_0) + 
\sum_{i=1}^2 \|G_i\|_{LE_{\Max}^{*,k+1}}.
  \end{equation}
\end{definition}

Similarly to the case of the scalar wave equation, the first
definition is adapted to the nontrapping case, while in the second we
allow for a loss of a derivative to account for possible trapped
geodesics.

We further comment on the choice of weights for the radial components,
by first noting that for radial metrics the radial components uncouple
and satisfy an equation which is essentially of the form $d (r^2 \bar
F) = r^2 \bar G$, and the charge at infinity is the limit of $r^2 \bar
F$.  Our weights require $\bar F$ to decay at least like $r^{-2}$ at
infinity, and $\bar G$ to decay at least like $r^{-3}$ at infinity,
with added integrability. Given the form of the radial equation,
this exactly suffices in order to guarantee that the charge remains zero at infinity,
and that $r^2 \bar F$ can be obtained by integrating $r^2 \bar G$ from infinity.

We also need an estimate similar to \eqref{sle} for the Maxwell
tensor. At least for stationary metrics it is clear that \eqref{sle}
is equivalent to a resolvent bound near zero frequencies.  As it turns
out, for our purposes here it is actually more efficient to work
directly with a zero frequency bound, even though our metric is
allowed to depend on time. We note that one could also harmlessly
carry out a similar substitution in the approach in \cite{MTT} for the
scalar wave equation, using the appropriate zero resolvent bound as
stated in \cite{Tat}. By analogy, we will refer to the estimate we
need as {\em the zero resolvent bound} for the Maxwell equation. To
state it we consider the fixed time operator $d^0$, acting on 2-forms,
which is obtained from $d$ by eliminating the time derivatives. In
other words, we define $d^0$ so that
\begin{equation}\label{d0def}
  d^0 F = dF - dt\wedge\L_{\partial_t} F.
\end{equation}
Then we consider the fixed time system
\begin{equation}\label{Maxwfixedt}
  d^0 F = G^0_1, \qquad d^0 * F = G^0_2. 
\end{equation}
 
\begin{definition}\label{MEresbound}
  We say that the problem \eqref{Maxwelleqns} satisfies the zero
  resolvent bound if on any time slice $t=t_0$ and for any $k \geq 0$, the
  system \eqref{Maxwfixedt} satisfies the following estimate:
  \begin{equation}\label{zero-res}
    \| F(t_0) \|_{\LE_{\Max}^{k}} \lesssim \sum_{i=1}^2 \|G^0_i (t_0)\|_{\LE_{\Max}^{*, k}}
  \end{equation}
for all $F$ so that the norm on the left is finite, and, in addition,
the following decay condition holds at infinity:
\begin{equation}\label{inf-bc}
\lim_{R \to \infty} \|1_{r > R} r \bar F(t_0)\|_{\LE} = 0.
\end{equation}
\end{definition}

We note that only the translation vector fields $\partial$ are used in
\eqref{Mderivloss} and \eqref{zero-res}. As part of our result, we
will prove that similar bounds hold for the vector fields $\Omega$ and
$S$. We also remark that for stationary metrics the bound \eqref{zero-res} follows from the
local energy decay estimates, in the same manner as in \cite{Tat}.

The requirement that $\bar F$ satisfies \eqref{inf-bc} is critical in
order to fix the charges to zero at infinity.

\newsection{The main result} 

For comparison purposes, we first state the similar result for the
scalar wave equation \eqref{box}, which was proved in \cite{MTT}:
 
 \begin{theorem}\label{mainscalar}
   Let $g$ be a metric which satisfies the conditions $(i)$, $(ii)$,
   $(iii)_a$ or $(i)$, $(ii)$, $(iii)_b$.  Assume that weak local energy decay and stationary local
   energy bounds hold for solutions to the wave equation \eqref{box}.
   Suppose $(u_0,u_1)$ and $f$ are supported inside the cone $C = \{ t
   \geq r - R_1\}$ for some $R_1>0$. Then for any fixed multi-index
   $\Lambda$ the following estimate holds in normalized coordinates
   for a large enough $m$:
   \begin{equation}
     | u^{\Lambda}(t,x)| \lesssim \kappa \frac{1}{\la t\ra \la t-r\ra^{2}}, 
\qquad | \nabla u^{\Lambda}(t,x)| 
     \lesssim \kappa \frac{1}{\la r \ra \la t-r\ra^{3}}
     \label{pointwiseestwave}
   \end{equation}
   where
   \[
   \kappa = \|\nabla u(0)\|_{H^m} + \| t^\frac52 f^{\leq m}\|_{LE^*}
   + \|\la r\ra t^\frac52 \nabla f^{\leq m}\|_{LE^*}.
   \]

 \end{theorem}

 We are now ready to state the main result of the paper. Consider the
 frame $(\partial_u, \partial_v, e_A, e_B)$, where as usual we set
 \[
 u = t-r, \qquad v = t+r
 \]
 and $(e_A, e_B)$ is an orthonormal frame of the unit sphere $(\S^2, d\omega)$. We have:
 
 \begin{theorem}\label{main}
   Let $g$ be a metric which satisfies the conditions $(i)$, $(ii)$,
   $(iii)_a$ or $(i)$, $(ii)$, $(iii)_b$.
Assume that the evolution
   \eqref{Maxwelleqns} satisfies the weak local energy bounds
   \eqref{Mderivloss} and the zero resolvent bound from Definition
   \ref{MEresbound}.  Moreover, let $F(0)$ and $G$ be supported inside
   the cone $C = \{ t \geq r - R_1\}$ for some $R_1>0$, and let $F$
   solve \eqref{Maxwelleqns}.  Then the following peeling estimates
   hold in normalized coordinates for large enough $m$:
   \begin{equation}
     \begin{split}
       | F_{uA}| \lesssim \kappa \frac{1}{\la t\ra \la t-r\ra^{3}} \\
       | F_{uv}| \lesssim \kappa \frac{1}{\la t\ra^{2} \la t-r\ra^{2}} \\
       | F_{AB}| \lesssim \kappa \frac{1}{\la t\ra^{2} \la t-r\ra^{2}} \\
       | F_{vA}| \lesssim \kappa \frac{1}{\la t\ra^3 \la t-r\ra}
     \end{split}
     \label{pointwiseest1}
   \end{equation}
   where
   \[
   \kappa = E^{m}(0) + \sum_{i=1}^2 \Bigl(\| t^{\frac72} \la r\ra^{-1} G_i^{\leq m}\|_{LE^*} + \| t^{\frac72} \la r\ra \overline {G_i}^{\leq m}\|_{LE^*}\Bigr).
   \]
Similar bounds will also hold for $F^{\Lambda}$ for $|\Lambda| \ll m$.
 \end{theorem}

It is useful at this point to review  the situations where we already 
know that the hypothesis of the  theorem is verified. So far, this is only 
the case for spherically symmetric black hole space-times, 
where we can use  the result of  \cite{StTat}; this is further discussed 
at the end of this section. Another interesting case where we are 
almost there is that of Kerr metrics with small angular momentum. 
There we have available the result of \cite{AB2}. Unfortunately this result 
only applies for solutions to the homogeneous Maxwell equation, 
so it cannot be applied directly. We do note that there are standard 
duality arguments which allow one to pass from the homogeneous to the 
ingomogeneous problem in local energy bounds. However, in the 
present situation such arguments would have to be adjusted to 
deal with charges.

 We further remark that in a compact spatial region we obtain the rate of
 decay of $t^{-4}$ for all components. This rate of decay is better
 than the rate of $t^{-\frac52}$ that was obtained, for Minkowski
 space times, in \cite{ChKl}. We also note
 that the $t^{-4}$ rate of decay for $F_{uv}$, $F_{AB}$ on
 Schwarzschild space-times was previously obtained in \cite{DSS, DSS2} by
 making heavy use of the stationarity and radial symmetry of the
 problem.
 
 On the other hand, we note that various components (layers) of $F$
 expressed in the null frame are decaying at different rates along
 outgoing null cone. This type of behavior is known as peeling
 estimates and has been first observed in the physics literature in
 \cite{Sac}, \cite{Pen}. For the Minkowski space-time, peeling
 estimates are known, see for example \cite{ChKl}, and similar results
 have been obtained for Schwarzschild space-times, see for instance
 \cite{IngNic} and \cite{MasNic}.  See also the related results
 \cite{Blue}, \cite{DSS, DSS2}, \cite{FS}, \cite{G} for decay
 estimates for Maxwell fields on Schwarzschild geometries.

Finally, we reemphasize the role played by the improved decay assumptions (precisely
by an additional factor of $r$) on the radial part of the source term $G$.  This guarantees 
that our solutions effectively behave as zero charge solutions, and the residual charge 
inside the cone plays only a perturbative role in the analysis. As part of our analysis,
we obtain a better decay rate for the radial part of the Maxwell field, namely 
\begin{equation}\label{pointwiseest2}
| \bar F| \lesssim \kappa \frac{1}{\la t\ra^2  \la r\ra \la t-r\ra^{2}}
\end{equation}

The rest of the paper is dedicated to the proof of Theorem
\ref{main}. In Section 4 we supplement the local energy estimates
\eqref{Mderivloss} and the zero resolvent bounds
\eqref{zero-res}, which are assumed to hold only for the translation
vector fields $\partial$, with similar estimates involving $\Omega$ and $S$.
Section 5 is dedicated to obtaining zero resolvent bounds with
different weights at infinity. Section 6 contains an improvement on
the bounds for the radial part of the tensor. Finally, Section 7 is
the main part of the proof and is divided into two parts. In the
first part we treat the Maxwell system as a system of wave equations
and mimic the proof of the main result in \cite{MTT} to get the rates
of decay \eqref{pointwiseestwave} for all components of the tensor
$F$.\footnote{This is where we use the estimates from Sections 5 and
  6.} In the second part, we use the Maxwell system to improve the
rates of decay and obtain the peeling estimates \eqref{pointwiseest1}.

\subsection{ The case of spherically symmetric metrics}

Here we provide a brief discussion of the spherically symmetric case, which 
is discussed in \cite{StTat}. The situation considered there is that of spherically 
symmetric black hole space times with the following two properties:

\begin{itemize} 
\item The event horizon is nondegenerate.

\item The trapped set (photon sphere) is unique, and strictly hyperbolic.
\end{itemize} 

For such space-times we have:

\begin{proposition} 
The hypothesis of Theorem~\ref{main} is satisfied for spherically symmetric black hole space
times as in \cite{StTat}.
\end{proposition}

\begin{proof}[Outline of proof]
The steps of the proof are as follows:
\medskip

1. Uncouple the radial and nonradial components $\bar F$ and $F - \bar F$.

\medskip

2. For the radial components the equations take essentially the form
\[
d (r^2 \bar F) = r^2 \bar J,
\]
see the equations (1.11) in \cite{StTat}. Thus all bounds for the 
radial components $\bar F$ are obtained by direct integration from infinity,
see also Remark 1.6 in \cite{StTat}. 

  \medskip

  3.  The local energy bound \eqref{Mderivloss} holds for the
  nonradial part with $k=0$; this is the main result of \cite{StTat},
  Theorem 1.3.  In effect the result in there is more akin to
  \eqref{Mderivno} with $k=0$, with a loss localized to the trapped
  set, i.e. the photon sphere.

\medskip

4. The local energy bound  \eqref{Mderivloss} holds for the nonradial part with $k \geq 1$.
This is by now a fairly standard argument, using  the red shift property on the horizon, and only elliptic analysis
away from it. This mirrors prior work of various authors for the scalar wave equation, see for instance
\cite{DaRoNotes}, \cite{TT}.  More precisely, one can add derivatives to the estimates as in the proof 
of Theorem 4.4 in \cite{TT}.

\medskip

5. The zero resolvent bound in Definition \ref{MEresbound} holds. This
follows by Plancherel's theorem from the $k=0$ form of the local
energy decay bound \eqref{Mderivloss}, by an argument similar to the
one in \cite{Tat} for the scalar wave equation.
\end{proof}

\newsection{ Vector field estimates.}

As stated in \eqref{Mderivloss} and \eqref{zero-res}, both the local
energy decay and the zero resolvent bounds are assumed to hold for the
derivatives of $F$.  Our goal here is to extend these properties
to the full set of vector fields $Z$, i.e. including the rotations $\Omega$ 
and scaling $S$, applied  the Maxwell field $F$. The result is summarized
in the following  lemma:
 
 \begin{lemma}
   Assume that weak local energy decay and the zero resolvent bound,
   \eqref{Mderivloss} and \eqref{zero-res}, hold.  Then we also have
   \begin{equation}
     \sup_{t > t_0} E[F^{\leq m}](t) +   \| F^{\leq m}\|_{LE_{\Max}}
     \lesssim E^{1}[F^{\leq m}](t_0) + \sum_{i=1}^2 \|G_i^{\leq m+1}\|_{LE_{\Max}^{*}}.
     \label{lew:vf}
   \end{equation}

\begin{equation}
  \| F^{\leq m}(t_0) \|_{\LE_{\Max}} \lesssim \sum_{i=1}^2 \|G_i^{0,\leq m} (t_0)\|_{\LE_{\Max}^{*}}.
  \label{sle:vf}
\end{equation}
\label{l:levf}\end{lemma}
\begin{proof}

  We begin with \eqref{lew:vf}. Note that for any vector field $X$
  and $F$ satisfying \eqref{Maxwelleqns}, we have
  \[
  d(\L_X F) = \L_X G_1, \qquad d*(\L_X F) = \L_X G_2 + H,
  \]
  where
  \begin{equation}\label{Hdef}
    H = d([*, \L_X] F).
  \end{equation}

  We now need to commute the Lie derivative with the Hodge star. By
  using the well-known formulas
  \begin{equation}\label{Lieform}
  (\L_X F)_{\alpha\beta} = X^{\gamma}\partial_{\gamma}F_{\alpha\beta}
  + F_{\gamma\beta}\partial_{\alpha}X^{\gamma} +
  F_{\alpha\gamma}\partial_{\beta}X^{\gamma},
  \end{equation}
  \begin{equation}\label{Hodgeform}
  (*F)_{\alpha\beta} =\frac{1}{2}
  \epsilon_{\alpha\beta\gamma\delta}\sqrt{-g}g^{\gamma\mu}g^{\delta\nu}F_{\mu\nu},
  \end{equation}
  we easily obtain
  \begin{equation}\label{commut}
    \begin{split}
      ([*, \L_X]F)_{\alpha\beta} = & - \frac{1}{2}
      X(\epsilon_{\gamma\delta\alpha\beta}\sqrt{-g}g^{\gamma\mu}g^{\delta\nu})F_{\mu\nu}
      + \frac{1}{2}
      \epsilon_{\gamma\delta\alpha\beta}\sqrt{-g}g^{\gamma\mu}g^{\delta\nu}
      (F_{\rho\nu} \partial_{\mu}X^{\rho} +
      F_{\mu\rho} \partial_{\nu}X^{\rho}) \\ & -\frac{1}{2}
      \sqrt{-g}g^{\gamma\mu}g^{\delta\nu}
      F_{\mu\nu}(\epsilon_{\gamma\delta\rho\beta} \partial_{\alpha}X^{\rho}
      + \epsilon_{\gamma\delta\alpha\rho} \partial_{\beta}X^{\rho}).
    \end{split}
  \end{equation}

  If $X\in\Omega$ we obtain that
  \begin{equation}\label{omega*commut}
    [*, \L_X]F \in S^Z(r^{-2})(F),
  \end{equation} 
  and thus also
  \begin{equation}\label{omegacommut}
    H \in S^Z(r^{-3})(F) + S^Z(r^{-2})(\L_{\partial} F).
  \end{equation}  
Here \eqref{omega*commut} follows from \eqref{commut}, the
 fact that the commutator vanishes for spherically symmetric metrics
 (since $\Omega$ would then be a Killing vector field) and the
 condition (ii) on the metric $g$.  We note that \eqref{omegacommut}
 uses the hypothesis \eqref{srderiv} but that is not strictly required
 for the proof of this lemma as it suffices to have $r^{-2}$-type decay on the first term.
  
  Unfortunately this is not quite enough to close the
  argument. Indeed, we would like to prove that
  \begin{equation}\label{omegale}
   \sup_{t > t_0} E[\L_X F](t) + \|\L_X F\|_{LE_{\Max}} \lesssim E^{1}[F^{\leq 3}](t_0) + \sum_{i=1}^2 \|G_i^{\leq 4}\|_{LE_{\Max}^{*}}.
  \end{equation}

 A first computation, using \eqref{Mderivloss},  gives
  \[
   \sup_{t > t_0} E[\L_X F](t) + \|\L_X F\|_{LE_{\Max}} \lesssim E^{1}[F^{\leq 3}](t_0) + \sum_{i=1}^2 \|\L_X G_i^{\leq 1}\|_{LE_{\Max}^{*}} + \|H\|_{LE_{\Max}^{*,1}} 
\]
while \eqref{omegacommut} combined with \eqref{Mderivloss} yields
\[ 
\|H\|_{LE_{\Max}^{*,1}} 
\lesssim E^{1}[F^{\leq 3}](t_0) + \sum_{i=1}^2 \|G_i^{\leq 4}\|_{LE_{\Max}^{*}} + 
\|\la r\ra\overline{H}\|_{LE^{*,1}}. 
  \]
  
  We would like to combine the last two bounds.  This almost works,
  except for the radial part $\overline{H}$; indeed, a priori one can
  only estimate
  \[
  \|\la r\ra\overline{H}\|_{LE^{*,1}} \lesssim \|\la r\ra^{-1} F\|_{LE^{*,2}} 
\]
  The  term on the right is not controlled by \eqref{Mderivloss} (though the failure is
  only logarithmic).  To avoid this issue, we remove this bad term by
  introducing a correction $\tilde F$ of $\L_X F$ as follows:
  \begin{equation}\label{tildeH}
    *\tilde F = (\overline{[*, \L_X] F})_{\phi\theta}d\omega^2
  \end{equation}

 Clearly by \eqref{omega*commut}
\begin{equation}\label{tildeFest}
 \tilde F \in S^Z(r^{-2}) F.
\end{equation}

 Thus 
 \[
\sup_{t > t_0} E[\tilde F](t) + \|\tilde F\|_{LE_{\Max}} \lesssim \sup_{t > t_0} E[F](t) + \|F\|_{LE_{\Max}} 
\]
with room to spare, so it is enough to prove the bound 
\[
\sup_{t > t_0} E[\L_X F - \tilde F](t) + \|\L_X F - \tilde F\|_{LE_{\Max}} \lesssim E^{1}[F^{\leq 3}](t_0) + \sum_{i=1}^2 \|G_i^{\leq 4}\|_{LE_{\Max}^{*}}
\] 
  We have
  \[
  d*(\L_X F - \tilde F) = \L_X G_2 + H - d*\tilde F,
  \]
 Since $d$ annihilates $(\overline{[*, \L_X] F})_{tr}dt\wedge dr$ and due to our choice of $\tilde F$, the difference $H - d*\tilde F$ has no radial mode
and can be estimated by
  \[
  \|H - d*\tilde F\|_{LE_{\Max}^{*,1}} \lesssim \|F\|_{LE^2} 
  \lesssim E^3(t_0) + \sum_{i=1}^2 \|G_i\|_{LE_{\Max}^{*,3}}
  \]
  On the other hand, we have 
\[
d(\L_X F - \tilde F) =  \L_X G_1 - d \tilde F
\]
so we need to bound the last term in $LE_{\Max}^{*,1}$. Again, one may be concerned
with the radial part. However, it is easy to see, using the
  asymptotic flatness of the metric, that
  \[
  \tilde F - (\overline{\tilde F})_{tr} dt\wedge dr  \in S^Z(r^{-1}) \tilde F 
  \]
  Hence we obtain the favorable expression
  \[
d \tilde F  \in S^Z(r^{-1}) \L_{\partial} \tilde F + S^Z(r^{-2}) \tilde F 
 \]
which, taking \eqref{tildeFest} into account, suffices in order to estimate $\|d \tilde F\|_{LE_{\Max}^{*,1}}$ by $\|F\|_{LE^2}$. This completes the proof of \eqref{omegale}.

Next we turn our attention to the scaling vector field $S$.
With $H$ as in \eqref{Hdef}, it is enough to prove that
  \begin{equation}\label{Scommut}
    H \in S^Z(r^{-2})(F, \L_{\{\partial,\Omega\}} F) + S^Z(r^{-1})G_2
  \end{equation}
  
  The same arguments as above will then yield the analogue of \eqref{omegale}, namely
  \begin{equation}\label{Scale}
   \sup_{t > t_0} E[\L_S F](t) + \|\L_S F\|_{LE_{\Max}} \lesssim E^{1}[F^{\leq 9}](t_0) + \sum_{i=1}^2 \|G_i^{\leq 10}\|_{LE_{\Max}^{*}}.
  \end{equation}
  
  We immediately get that $H \in S^Z(r^{-1})(F^{\leq 1})$ by
  \eqref{commut} and the fact that $S$ is a conformal Killing vector
  field for the Minkowski metric. We also note that since $g_{sr}\in
  S^Z(r^{-2})$, it is enough to prove \eqref{Scommut} for the
  spherically symmetric part $\tilde g = m + g_{lr}$, which in
  normalized coordinates can be written (see \eqref{normlr}):
  \begin{equation}\label{tildeg}
  \tilde g = -dt^2 + dr^2 + r^2 (1+ g_{\omega}(r))d \omega^2, \qquad
  g_\omega \in S^Z_{rad}(r^{-1}).
  \end{equation}
  Note in particular that $\tilde g$ is diagonal in the $(t, r, \phi, \theta)$ coordinates.
 
  A careful inspection of \eqref{commut} reveals that
  \begin{equation}\label{S*commut}
 \tH_{\alpha\beta} :=  ([*_{\tilde g}, \L_S]F)_{\alpha\beta} = \epsilon_{\gamma\delta\alpha\beta}(-S(\sqrt{-\tilde g}\ \tilde g^{\gamma\gamma}\tilde g^{\delta\delta} )+ \kappa\sqrt{-\tilde g}\ \tilde g^{\gamma\gamma} \tilde g^{\delta\delta}) F_{\gamma\delta}
  \end{equation}
  where $(\alpha,\beta, \gamma, \delta)$ is some permutation of $(t,
  r, \phi,\theta)$ and
  \[
  \kappa = \left\{ \begin{array}{ccc} -2 & (\alpha,\beta)=(t,r) , \cr
      2 & (\alpha,\beta)=(\phi,\theta), \cr 0 & \qquad \text{otherwise}.
    \end{array}
  \right.
  \]
 We remark that, due to \eqref{tildeg}, we have
 \[
 \tH_{t\theta}= \tH_{t\phi} = \tH_{r\theta} = \tH_{r\phi} = 0.
 \]

  We now take the exterior derivative of the tensor $\tH$, and subsequently pass to the $(t, r, A, B)$ frame. Every time
  a derivative falls on the metric coefficients, we gain a factor of
  $r^{-2}$. Since $e_{A,B} = \frac1r \Omega$, we obtain
  \[
  H_{tAB} \approx r^{-2} \partial_{t} \tH_{\phi\theta} \in S^Z(r^{-1})\partial_{t} F_{tr} +
  S^Z(r^{-2})(F, \L_{\{\partial, \Omega\}}F)
  \]
  \[
  H_{rAB} \approx r^{-2} \partial_{r} \tH_{\phi\theta} \in S^Z(r^{-1})\partial_{r} F_{tr} +
  S^Z(r^{-2})(F, \L_{\{\partial, \Omega\}}F)
  \]
 Let us now notice that the second equation in \eqref{Maxwelleqns}
  implies that
  \[
  \partial_{t,r} F_{tr} \in S^Z(1)(G_2) + S^Z(r^{-1})(F, \L_{\{\partial,
    \Omega\} } F).
  \]
On the other hand,
  \[
  H_{tr\phi} = \partial_{\phi} \tH_{tr} \in S^Z(r^{-3}) \partial_{\phi}F_{\phi\theta}  \]
  and similarly for $H_{tr\theta}$. This implies
 \[
 H_{trA}, H_{trB} \in S^Z(r^{-2})(F, \L_{\{\partial, \Omega\} } F).
 \] 
Thus \eqref{Scommut} for $X=S$ is now proved. 
  
   Since $\L_X \bar F = \overline {\L_X F}$ for $X\in\{\Omega, S\}$, \eqref{omegale} and \eqref{Scale} imply the local energy decay bound for $\L_{\Omega} F$ and for $\L_S F$. More derivatives can be readily added to our argument, and higher powers of $\Omega$ and $S$ are dealt with by induction.

  The proof of \eqref{sle:vf} is similar. Note that for any vector
  field $X$ and $F$ satisfying \eqref{Maxwfixedt}, we have
  \[
  d^0(\L_X F) = \L_X G^0_1 + [\L_X, d-d^0]F, \qquad d^0*(\L_X F) = \L_X
  G^0_2 + H + [\L_X, (d-d^0)*]F,
  \]
  with $H$ given by \eqref{Hdef}.

  One now easily checks, using \eqref{d0def} and the fact that $\L_{\partial_t} F =0$ as $F$ is frozen at time $t_0$, that $[\L_X, d-d^0]F=0$
  if $X\in\{\Omega, S\}$. Moreover,
  \[ [\L_X, (d-d^0)*]F = dt\wedge\L_{\partial_t} [*, \L_X] F
  \]
  where the Hodge star is frozen at time $t_0$.
  
  The proofs of the analogues of \eqref{omegale} and \eqref{Scale}, namely 
  \[
  \|\L_{\Omega} F(t_0)\|_{\LE_{\Max}} \lesssim \sum_{i=1}^2 \|G_i^{0,\leq 3} (t_0)\|_{\LE_{\Max}^*},
  \]
  \[
  \|\L_{S} F(t_0)\|_{\LE_{\Max}} \lesssim \sum_{i=1}^2 \|G_i^{0,\leq 9} (t_0)\|_{\LE_{\Max}^*},
  \]
  follow from using \eqref{zero-res} and the same methods as
  above.
  The lemma follows by induction.
\end{proof}

\newsection{ Elliptic zero resolvent bounds.}

 The zero resolvent bound from \eqref{sle:vf} can be viewed more as a qualitative 
statement about the absence of zero eigenvalues and resonances (except for the 
charge induced modes, which we asymptotically identify with the radial part of 
$F$). Because of this, one has a choice over the weights that are used at infinity,
very much like in the similar estimates for the inverse Laplacian. This idea is explored
in this section and will play a key role in obtaining the correct pointwise decay  estimates 
 in a bounded region.  Our main result is as follows:

\begin{lemma}
  The following fixed time estimates for solutions to \eqref{Maxwfixedt}, 
restricted to fields $F \in \LE_{\Max}$ with the additional
property \eqref{inf-bc},  are equivalent:
  \begin{equation}\label{zero-res0} \begin{split}
    \| F^{\leq m}(t_0) \|_{\LE} 
+  \|\la r\ra \overline{F}^{\leq m}(t_0)\|_{\LE}  \lesssim 
    \sum_{i=1}^2 \|G^{0,\leq m}_i (t_0)\|_{\LE^*} + \|\la r\ra \overline{G_i^0}^{\leq m} (t_0)\|_{\LE^*},
\end{split}  \end{equation} 
  \begin{equation}\label{zero-res1} \begin{split}
    \| \la r\ra^{-1} F^{\leq m}(t_0) \|_{\LE} 
+ \|\la r\ra \overline{F}^{\leq m}(t_0)\|_{\LE}   \lesssim 
    \sum_{i=1}^2  \|\la r\ra^{-1} G^{0,\leq m}_i  (t_0)\|_{\LE^*} + \|\la r\ra \overline{G_i^0}^{\leq m} (t_0)\|_{\LE^*}.
\end{split}  \end{equation}
  \label{impl:levf}\end{lemma}

We remark that \eqref{zero-res0} is the same as \eqref{sle:vf}, so it holds true under the assumption \eqref{zero-res} due to Lemma~\ref{l:levf}. 

We note that the above estimates are required to hold whenever $F$ has the regularity 
stated in the beginning, and the right hand side is finite. The a priori regularity of 
$F$ is needed in order to preclude the existence of solutions to the homogeneous 
$d^0$ system. We will only use these for compactly supported $F$, but for the proof 
it is more convenient to work with a weaker a priori decay assumption \eqref{inf-bc}.
 We also remark that this is not a solvability property, it is just an a priori bound.

\begin{proof}

  We will first prove the lemma for $m=0$. In order to simplify the
  notation, since all the analysis takes place on a $\{t=t_0\}$ slice,
  we will drop $t_0$ for the rest of the proof.  The weights in the
  two estimates are comparable in a compact set. Thus, the proof is
  primarily concerned with the analysis at infinity. But at infinity
  our problem is reasonably well approximated by the Minkowski
  problem, so for the most part it suffices to do a perturbative
  analysis. We begin with a brief analysis of what happens in the
  Minkowski space-time.

{\bf The Minkowski case.} 
Denoting by $\bigstar$  the Hodge star of the Minkowski metric,
the Minkowski equation has the form
 \begin{equation}\label{Minkpart}
    d^0 F =  G^0_1, \qquad d^0 \bigstar F =  G^0_2 
  \end{equation} 
where $d^0$ is now the standard exterior differentiation on a fixed time slice.
We will prove both \eqref{zero-res0} and \eqref{zero-res1} in the Minkowski case
by separating the radial and nonradial parts. We remark that our proof
also gives the recipe for constructing the unique solution $F$ which satisfies
 \eqref{inf-bc}.

For the radial parts we have
\[
  \partial_r (r^2 \overline{F_{AB}}) = r^2
  \overline{(G^0_1)}_{rAB}, \qquad \partial_r (r^2 \overline{
    F^\bigstar_{AB}}) = r^2 \overline{(G^0_2)}_{rAB}
 \]
where $F^\bigstar = \bigstar F$.
 The decay condition \eqref{inf-bc} at infinity allows us integrate
 these equations from infinity to uniquely determine the components
 $\overline{F_{AB}}$ and $\overline{ F^\bigstar_{AB}}$.  Outside a
 ball these will satisfy the straightforward bound
\begin{equation}\label{Mtf-rad}
\| r(\overline{F_{AB}},\overline{ F^\bigstar_{AB}})\|_{\LE} 
+ \| r\nabla (\overline{F_{AB}},\overline{ F^\bigstar_{AB}})\|_{\LE} 
\lesssim  \| r(\overline{G^0_1},\overline{G^0_2})\|_{\LE^*}. 
\end{equation}
We remark that in the Minkowski case the boundary condition at infinity will 
in general force an $r^{-2}$ blow-up at zero for the radial part. In our case 
this does not happen because of our a-priori assumption $F \in \LE_{\Max}$.

For the nonradial part we argue in a more standard manner. For any tensor $A$, let $A_{nr}= A-\bar A$. We can rewrite \eqref{Minkpart} as  
\begin{equation}\label{Minkwave}
    \Delta_x F_{nr, \alpha\beta} =   ( \bigstar d^0 \bigstar (  G^0_{1,nr}) +
 d^0 \bigstar ( G^0_{2,nr}))_{\alpha\beta} 
  \end{equation}
  where $\Delta_x$ is the usual Euclidean Laplacian. This is solved
in the standard manner, using the fundamental solution for the 
Laplacian. Then the estimate
  \begin{equation}\label{Minkzero-res1/2}
    \| r^{-1} F_{nr}\|_{L^2} + \| \nabla_x F_{nr}\|_{L^2} \lesssim \sum_{i=1}^2 \| G^{0}_{i,nr} \|_{L^2}
  \end{equation}
  is a direct consequence of the direct elliptic estimate for
  $\nabla_x F$, coupled with Hardy's inequality to get the bound for
  $F$.

To prove either \eqref{zero-res0} or \eqref{zero-res1} it
  suffices to start with $G^0_i$ supported in a fixed dyadic region
  $A_R$.  
  
  Let us start with \eqref{zero-res1}.  Clearly \eqref{Minkzero-res1/2} already suffices when $|x|>\frac{R}8$. When $|x|<\frac{R}8$, on the other hand, $F_{nr}$ is harmonic, so we
  trivially obtain the pointwise bound
\[
|F_{nr}(x)| + R |\nabla_x F_{nr}(x)|  \lesssim R^{-\frac12}  
(\| r^{-1} F_{nr}\|_{L^2} + \| \nabla_x F_{nr}\|_{L^2}) 
\]
which proves \eqref{zero-res1} for $|x|<\frac{R}8$.  

 For \eqref{zero-res0}, note first that \eqref{Minkzero-res1/2} suffices when $|x|<8R$. For $|x|>8R$ we note that we can write \eqref{Minkwave} as
 \[
 \Delta_x F_{nr, \alpha\beta}(x) = \sum_{i=1,2} \sum c_{ij\alpha\beta} \partial_j (G^0_{i, nr})_{\alpha\beta}
 \]
for some constants $c_{ij\alpha\beta}$. By using the fundamental solution of the Laplacian and the support property of $G^0_i$ we obtain the pointwise bound
\[
|F_{nr}(x)| + |x| |\nabla F_{nr}(x)| \lesssim \frac{R}{|x|^2} \sum_{i=1}^2 \| G^{0}_{i,nr} \|_{\LE^*},
\]
which immediately implies \eqref{zero-res0}.

We further observe that in the Minkowski case we have actually
proved a strengthened form of \eqref{zero-res0} and \eqref{zero-res1},
which includes gradient bounds on the left:
 \begin{equation}\label{Mzero-res0} \begin{split}
    \| F\|_{\LE} + \|\la r\ra \nabla_x F \|_{\LE} 
+  \|\la r\ra \bar F\|_{\LE}   + \|\la r\ra^2 \nabla_x \bar F\|_{\LE} \lesssim 
     \|G^{0}\|_{\LE^*} + \|\la r\ra \overline{G^{0}} \|_{\LE^*},
\end{split}  \end{equation} 
  \begin{equation}\label{Mzero-res1} \begin{split}
    \|\la r\ra^{-1} F\|_{\LE} + \| \nabla_x F \|_{\LE} 
+  \|\la r\ra \bar F\|_{\LE}   + \|\la r\ra^2 \nabla_x \bar F\|_{\LE} \lesssim 
     \|\la r\ra^{-1} G^{0}\|_{\LE^*} + \|\la r\ra \overline{G^{0}} \|_{\LE^*}.
   \end{split} \end{equation} 
By standard elliptic estimates, similar
 gradient terms can be added on the left in \eqref{zero-res0} and
 \eqref{zero-res1} in the nontrapping case. However, in the black hole
 case this can be done only outside a ball, more precisely in the
 region where $\partial_t$ is time-like.

\bigskip

{\bf The general case as a perturbation of Minkowski.}
Starting with the equation  \eqref{Maxwfixedt}, we write it as a perturbation 
of the Minkowski problem \eqref{Minkpart} as follows:
\begin{equation}\label{Mink-pert}
 d^0 F =  G^0_1 , \qquad d^0 \bigstar F 
=  G^0_2 +  d^0 (\bigstar - *)F.
\end{equation}
In order to work with this, we need to understand the size of the terms in the last 
expression. Our asymptotic flatness assumptions provide the following expansion:
\begin{equation}\label{errinfty}
d^0  (\bigstar - *)  F \in S^Z(r^{-1}) \nabla_x   F+ S^Z(r^{-2})   F,
\end{equation}
while for the radial part,
\begin{equation}\label{raderrinfty}
\overline{d^0  (\bigstar - *)  F} \in S^Z(r^{-1}) \nabla_x  \overline {  F}
+ S^Z(r^{-2}) (\nabla_x  F + \overline { F})
+ S^Z(r^{-3})    F.
\end{equation}
Next we use the Minkowski analysis above to deal with the general case.
We need two slightly
  different arguments in order to go up and down in terms of decay
  rates.
 
\bigskip

\def\tF{{\tilde F}}
{\bf The proof of $\eqref{zero-res0} \implies \eqref{zero-res1}$}.
 The main idea is to peel off the far part of the solution to
\eqref{Maxwfixedt} using a simple parametrix. Precisely, 
it suffices to construct an approximate solution $\tF$ near infinity
 which satisfies the bound \eqref{zero-res1},
as well as the error estimate
\begin{equation}\label{tF}
\|d^0 \tF - G^0_1\|_{\LE_{\Max}^*} + \|d^0 * \tF - G^0_2\|_{\LE_{\Max}^*}
\lesssim RHS\eqref{zero-res1}.
\end{equation}
Then the desired bound \eqref{zero-res1} for $F$ follows by applying 
\eqref{zero-res0} to $F-\chi \tilde F$, where $\chi$ is a smooth radial 
cutoff function which selects the exterior of a large ball.

The simplest idea to construct an approximate solution for
\eqref{Maxwfixedt} near infinity would be to treat the far away part of the
equation \eqref{Maxwfixedt} as a perturbation of the Laplacian. This
would work in order to prove any intermediate bound between
\eqref{zero-res0} and \eqref{zero-res1}, but not \eqref{zero-res1};
this is because at the level of \eqref{zero-res1} the radial and
nonradial modes become strongly coupled. To remedy this, we solve
directly for the radial parts, and perturbatively only for the
nonradial components.

Precisely, the radial part of the equations \eqref{Maxwfixedt} yields
the equations
\[
  \partial_r (r^2 \overline{F}_{AB}) = r^2
  \overline{G^0_1}_{rAB}, \qquad \partial_r (r^2 \overline{
    F^*_{AB}}) = r^2 \overline{G^0_2}_{rAB}
 \]
where $F^*=*F$.
We integrate these equations from infinity to uniquely determine
the components $\overline{F_{AB}}$ and $\overline{ F^*_{AB}}$.
Outside a ball these will satisfy the straightforward bound
\begin{equation}\label{tf-rad}
\| r(\overline{F_{AB}},\overline{ F^*_{AB}})\|_{\LE} 
+ \| r\nabla (\overline{F_{AB}},\overline{ F^*_{AB}})\|_{\LE^*} 
\lesssim  \| r(\overline{G^0_1},\overline{G^0_2})\|_{\LE^*} 
\end{equation}
which is akin to the Minkowski bound \eqref{Mtf-rad}.

To define $\tF$, we first obtain its nonradial part $\tF_{nr}$ by
solving the  Minkowski space-time version \eqref{Minkpart} of our equations.  
As discussed above, this satisfies the bounds
\begin{equation}\label{tf-nr}
\| r^{-1} \tF_{nr}\|_{\LE} + \|\nabla \tF_{nr}\|_{\LE}
\lesssim    \sum_{i=1}^2  \|\la r\ra^{-1} G^0_{i,nr} \|_{\LE^*}.
 \end{equation}
 
Now we define the radial part of $\tF$ by requiring it to match the two 
radial components of $F$ directly computed above, namely
\[
\overline {\tF_{AB}} = \overline{F_{AB}}, \qquad  \overline {\tF^*_{AB}} = \overline {F^*_{AB}}.
\]
The first equation gives directly $\tF_{AB}$. From the second we compute
\begin{equation}\label{tf-def}
\overline{\tF_{tr}} \in S^Z_{rad}(1)  \overline {F^*_{AB}} + S_Z(r^{-2}) \tF_{nr}.
\end{equation}

Our construction above yields a field $\tF$ outside a large ball. By
\eqref{tf-rad}, \eqref{tf-nr} and \eqref{tf-def} it follows that $\tF$
satisfies the bound \eqref{zero-res1}.  Further, $\tF$ solves exactly
the first equation in \eqref{Maxwfixedt}, as well as the radial
component of the second equation in \eqref{Maxwfixedt}.  It remains to
estimate the nonradial error in the second equation.  Using the
asymptotic flatness of the metric, we see that this is given by
\[
\begin{split}
d^0 (\tF^*)_{nr} - G^0_{2,nr} = & \ ( d^0 \ast \tF_{nr}  +  d^0 \ast \bar{\tF} )_{nr} -  G^0_{2,nr} 
\\
 = & \  (d^0 (\ast-\bigstar) \tF_{nr})_{nr}  
\\
\in & \  ( S^Z(r^{-1}) \nabla \tF_{nr} +   S^Z(r^{-2})  \tF_{nr})_{nr}.
\end{split}
\]

We can bound this error using \eqref{tf-nr} to obtain 
\[
\|r (d^0 (\tF^*)_{nr} - G^0_{2,nr}) \|_{\LE} \lesssim \sum_{i=1}^2 \|\la r\ra^{-1}
G_i^0\|_{\LE^*}.
\]
This almost gives \eqref{tF}, up to a logarithmic divergence. However, our error $\tilde G^0_2 := G^0_2-d*\tilde{F}$
decays better than $G_2^0$ by a power of $r$, so in order to obtain \eqref{tF} it suffices to reiterate
once more the above construction.

  {\bf The proof of $\eqref{zero-res1} \implies \eqref{zero-res0}$}. We begin
 with the series of inequalities
\[
LHS( \eqref{zero-res1}) \lesssim RHS( \eqref{zero-res1}) \lesssim  
RHS( \eqref{zero-res0}). 
\]
As observed earlier, we can also obtain an elliptic bound for $\nabla
F$ outside a compact set. These estimates provide a weaker bound,
which nevertheless suffices within a compact set.  Hence a
straightforward localization argument, namely replacing $F$ with $\chi
F$, allows us to reduce the problem to the case when $F$ is supported in an exterior region $\{r \gtrsim R_2\}$.

But in this region we can replace $g$ by $m$ perturbatively.  We write the equation
\eqref{Maxwfixedt} as in \eqref{Mink-pert}, and apply the bound \eqref{Mzero-res0}
in the Minkowski setting. It remains to estimate the error $\| d^0(*-\bigstar) F\|_{\LE^*_{\Max}}$,
for which we use the expressions \eqref{errinfty} and \eqref{raderrinfty}. We obtain
\[ \begin{split}
\| d^0(*-\bigstar) F\|_{\LE^*_{\Max}} & \lesssim \| \la r\ra^{-1} \nabla_x F\|_{\LE^*} + \| \nabla_x \bar F\|_{\LE^*} + \|\la r\ra^{-2} F\|_{\LE^*_{\Max}} \\
& \lesssim R_2^{-1/2} (  \|\la r\ra \nabla_x F\|_{\LE} + \|\la r\ra^2 \nabla_x \bar F\|_{\LE} + \| F\|_{\LE_{\Max}}).
\end{split}\]

If $R_2$ is large enough then this term is perturbative in  \eqref{Mzero-res0}, and the proof 
of \eqref{zero-res0} is concluded.

This concludes the proof for $m=0$. Higher spatial derivatives are
easily introduced in the argument in an elliptic fashion. Finally, the
same arguments as in Lemma \ref{l:levf} apply for $\Omega$ and $S$.
 \end{proof}

\newsection{ Charges and bounds for the radial part}
 
As explained earlier, the radial part of the solution is a good
approximation of the charge near spatial infinity. In particular, we
expect it to have better bounds (assuming the sources $G_1$ and $G_2$
have good decay at infinity), and we also expect it to not propagate in
a dispersive fashion along the cone.  However, there is some degree of
freedom in our choice of coordinates, and thus in what we call the
radial part. Hence, within our setup, there is some degree of mixing 
between radial and nonradial. The next result shows that the 
nonradial effects on the radial part have size $r^{-2}$; thus, as expected, 
they are weaker near infinity and stronger in a compact set. This 
is in a nutshell the content of the next lemma, which will come in very handily 
when we seek to propagate bounds for the radial part inside the cone, without any 
crossing penalty.

\begin{lemma}\label{impradpart}
  The radial part of $F$ satisfies the improved estimate
  \begin{equation}\label{zero-res1rad}
    \|\la r\ra^{\frac32} \bar F^{\leq m}(t_0)\|_{L^2(A_R)} \lesssim \sum_{i=1}^2\|\la r\ra^2 \bar G^{\leq m}_i (t_0)\|_{\LE^*} + \|\la r\ra^{-\frac12}F^{\leq m}(t_0)\|_{L^2(A_R)}.
  \end{equation} 
\end{lemma}
\begin{proof}

 The estimate is obvious when $R\approx 1$. 
  When $R\gg 1$, we will use the original system \eqref{Maxwelleqns}, which in
  particular implies that
  \begin{equation}\label{radparteqn}
  \partial_r (r^2 \overline{F}_{AB}) = r^2 \overline{(G_1)}_{rAB},
  \qquad \partial_r (r^2(\overline{*F})_{AB}) = r^2 \overline{(G_2)}_{rAB}.
  \end{equation}

 Moreover, due to \eqref{Hodgeform} we have
\begin{equation}\label{radparteqn1}
(\overline{*F})_{AB} = (1+S_{rad}^Z (r^{-1}))\overline{F}_{tr} + S^Z(r^{-2})F,
\end{equation}
as well as 
\begin{equation}\label{radparteqn3}
\partial_r (\overline{*F})_{AB} = (1+S_{rad}^Z (r^{-1}))\partial_r \overline{F}_{tr} 
+ S_{rad}^Z (r^{-2})  \overline{F}_{tr} +  S^Z(r^{-2}) \partial_r F+  S^Z(r^{-3})  F.
\end{equation}

  After integrating from infinity and applying the Schwarz inequality for
  the terms involving $G_i$, we obtain the desired conclusion for
  $m=0$:
  \begin{equation}\label{zero-res1rad0}
    \|r^{\frac32} \bar F(t_0)\|_{L^2(A_R)} \lesssim \sum_{i=1}^2\|r^2 \bar G_i (t_0)\|_{\LE^*} + \|r^{-\frac12}F(t_0)\|_{L^2(A_R)}.
  \end{equation}

  We now need to commute with vector fields in $Z$. After commuting \eqref{radparteqn} with $\partial_t$ and $\partial_r$ we easily obtain that
 \[
  \|r^{\frac32}  \partial_{t,r} \bar F(t_0)\|_{L^2(A_R)} \lesssim \sum_{i=1}^2\|r^2 \bar G^{\leq 1}_i (t_0)\|_{\LE^*} + \|r^{-\frac12}F^{\leq 1}(t_0)\|_{L^2(A_R)}.
 \] 
  Since $\partial_{x_i} f(t,r) \in S^Z(1) \partial_r f(t,r)$ for any radially symmetric function $f$, the inequality above also holds for all derivatives.
  
  On the other hand, for a vector field $X\in\{\Omega, S\}$ we know that $\L_X \bar F = \overline{\L_X F}$. After applying $\L_X$ to \eqref{Maxwelleqns} we get
  \[
  \partial_r (r^2 (\L_X \bar F)_{AB}) = r^2 (\L_X
    \bar G_1)_{rAB}, \qquad \partial_r (r^2 (\overline{*\L_X F})_{AB} + r^2
  (\overline{[\L_X, *]F})_{AB}) = r^2 (\L_X \bar G_2)_{rAB}.
  \]
 Due to \eqref{Hodgeform} we have
 \[
  (\overline{*\L_X F})_{AB} = (1+S_{rad}^Z (r^{-1}))(\L_X \bar F)_{tr} + S^Z(r^{-2})(\L_X F).
 \] 
   Thus after integrating from infinity and applying H\"older's inequality for
  the terms involving $G_i$, we obtain
  \begin{multline*}
  \|r^{\frac32} \L_X \bar F(t_0)\|_{L^2(A_R)} \lesssim
  \sum_{i=1}^2\|r^2 \L_X G_i (t_0)\|_{\LE^*} + \|r^{-\frac12}\L_X F(t_0)\|_{L^2(A_R)} \\+
  \|r^{\frac32}(\overline{[\L_X, *]F})_{AB}\|_{L^2(A_R)}.
  \end{multline*}
  When $X\in\Omega$ we see from \eqref{commut} and
  \eqref{omega*commut} that
  \[
  |[\L_X, *]F| \lesssim r^{-2} |F|.
  \]
  On the other hand, by using \eqref{S*commut} and the fact that
  $g_{sr}\in S^Z(r^{-2})$, we see that
  \[
  |(\overline{[\L_S, *]F})_{AB}| \lesssim r^{-1}|\overline{F}| + r^{-2}
  |F|.
  \]
  We thus obtain in both cases
  \[
  \|r^{\frac32} \L_X \bar F(t_0)\|_{L^2(A_R)} \lesssim
  \sum_{i=1}^2\|r^2 \L_X G_i (t_0)\|_{\LE^*} + \|r^{-\frac12}\L_X F(t_0)\|_{L^2(A_R)}.
  \]
 
  We can now use induction to conclude that \eqref{zero-res1rad} holds
  for all $m$.
 \end{proof}


\newsection{Proof of the main result}

  The proof of the main theorem will be divided into two parts. We
  first mimic the approach used in \cite{MTT} to prove Theorem
  \ref{mainscalar} to obtain pointwise bounds which are similar to those 
in the scalar  case:
  \begin{equation}\label{wkdecay}
    | F^{\leq  n}_{\alpha\beta}| \lesssim \frac{\kappa_1}{\la t\ra \la t-r\ra^{2}},\qquad | \nabla F^{\leq  n}_{\alpha\beta}| \lesssim \frac{\kappa_1}{\la r\ra \la t-r\ra^{3}}
  \end{equation}  
  where
  \[
  \kappa_1 = E^{n+m}(0) + \sum_{i=1}^2 \| t^{\frac52} G_i^{\leq n+m}\|_{LE^*} + \| t^{\frac52} r \overline {G_i}^{\leq n+m}\|_{LE^*}. 
  \]
    We then use \eqref{wkdecay} combined with the Maxwell system to
  improve the decay near the cone to the peeling estimates
  \eqref{pointwiseest1}.

 \subsection{ The Maxwell system as a wave equation.} 
 We start by rewriting the Maxwell system as a system of wave
  equations for each component. We have
\begin{equation}
  \nabla_{[\alpha} F_{\beta\gamma]} = G_{1\alpha\beta\gamma}, \qquad \nabla^{\alpha} F_{\alpha\beta} = -*G_{2\beta}.
  \label{maxwell-inhom1}\end{equation}

Differentiating the first equation we get
\begin{equation}
  \nabla^\alpha\nabla_{\alpha} F_{\beta\gamma} 
  + [\nabla^\alpha,\nabla_{\gamma}] F_{\alpha\beta}
  + [\nabla^\alpha,\nabla_{\beta}] F_{\gamma\alpha} \ =  \nabla^\alpha G_{1\alpha\beta\gamma} - \nabla_{[\beta}
  *G_{2\gamma]}
  \ , \notag
\end{equation}
where we have used $\nabla^\alpha F_{\alpha\beta}= -*G_{2\beta}$ in the
second and third term. The commutators are curvature contributions,
and cleaning these up we get:
\begin{equation}
  \Box_{g} F_{\alpha\beta} - R_{\alpha}^{\ \ \gamma}F_{\gamma\beta}
  - R_{\beta}^{\ \ \gamma}F_{\alpha\gamma}
  + R_{\alpha\beta}^{\ \ \ \gamma\delta}F_{\gamma\delta}
  \ = \  \nabla^\gamma G_{1\gamma\alpha\beta} - \nabla_{[\alpha}
  *G_{2\beta]} \ . \label{wave_form1}
\end{equation}
Here $\Box_{g}$ is the covariant wave equation acting on two forms; we
would like to replace this with $\Box_g [F_{\alpha\beta}]$, the
scalar d'Alembertian applied to each component separately. For that we compute 
\begin{multline}\label{gtoMink}
  \nabla^\gamma\nabla_{\gamma} F_{\alpha\beta} =
  \nabla^\gamma \partial_\gamma[F_{\alpha\beta}] -
  g^{\gamma\delta}\Bigl[\Gamma^\sigma_{\gamma\alpha,\delta}F_{\sigma\beta} -
  \Gamma^\sigma_{\gamma\beta,\delta} F_{\alpha\sigma} - 4
  \Gamma^\sigma_{\delta[\beta} F_{\alpha]\sigma,\gamma} \\+2
  \Gamma^\sigma_{\delta\gamma} \Gamma^\mu_{\sigma[\beta}F_{\alpha]\mu}
  + 2\Gamma^\mu_{\gamma\sigma} \Gamma^\sigma_{\delta[\beta}F_{\alpha]\mu}\Bigr].
\end{multline}
Taking into account that $g$ is in normalized coordinates, as well as \eqref{srderiv} (which in
particular imply that $\Gamma_{\alpha\beta}^{\gamma} \in S^Z
(r^{-2})$ and $R_{\alpha\beta\gamma\delta} \in S^Z(r^{-3})$), one easily obtains that each component $F_{\alpha\beta}$
satisfies
\begin{equation}\label{compnts}
  \Box_g [F_{\alpha\beta}] = Q_{\alpha\beta} \in S^Z(1)(G_1^{\leq 1}, G_2^{\leq 1}) + S^Z(r^{-2}) \partial F + S^Z(r^{-3}) F.
\end{equation}
The decay on the last term follows from \eqref{srderiv}, and it is
here that this hypothesis is crucial.

This immediately implies the analogous equation for $\Box
F_{\alpha\beta}$. After commuting with the vector fields in $Z$ we
also obtain by induction for all multi-indices $\Lambda$:
\begin{equation}\label{boxua}
  \Box F^{\Lambda}_{\alpha\beta} \in S^Z(1)(G_1^{\leq |\Lambda|+m}, G_2^{\leq |\Lambda|+m}) + S^Z(r^{-2})(\partial F^{\leq |\Lambda|+m}) + S^Z(r^{-3})(F^{\leq |\Lambda|+m}).
\end{equation}
Here and in the sequel $m$ will be a large enough number that could
change from equation to equation.

We can now apply Lemma 3.10 from \cite{MTT} which gives a first
pointwise estimate in terms of the local energy norms:
\begin{equation}\label{frstest}
  |F^{\Lambda}_{\alpha\beta}|\lesssim \frac{\log \la t-r\ra}{\la r\ra \la t-r\ra^{\frac{1}{2}}}
  \Bigl(\sum_{\alpha,\beta}\| F^{\leq |\Lambda| +
    m}_{\alpha\beta}\|_{LE^1} + \sum_{i=1}^2 \|\la r \ra G_i^{\le |\Lambda|+m}\|_{LE^{*}}\Bigr).
\end{equation}

At this point of the proof we would like to analyze what happens
inside the cone (the region $r\ll t$) and near the cone (the region
$t\approx r$) separately. In the first region we will use the zero
resolvent bound for the Maxwell system, while in the second region we will 
fall back onto the wave equation analysis  and use the
fundamental solution for the Minkowski wave equation.

\subsection{Notations and localized Klainerman-Sobolev bounds}
We first recall some notation from \cite{MTT}.  For the forward cone $C
= \{ r \leq t+R_1\}$ we consider a dyadic decomposition in time into
sets
\[
C_{T} = \{ T \leq t \leq 2T, \ \ r \leq t+R_1\}.
\]
For each $C_T$ we need a further double dyadic decomposition of it
with respect to either the size of $t-r$ or the size of $r$, depending
on whether we are close or far from the cone,
\[
C_{T} = \bigcup_{1\leq R \leq T/2} C_{T}^{R} \cup \bigcup_{1\leq U <
  T/2} C_T^{U}
\]
where for $R,U > 1$ we set
\[
C_{T}^{R} = C_T \cap \{ R < r < 2R \}, \qquad C_{T}^{U} = C_T \cap \{
U < t-r < 2U\}
\]
while for $R=1$ and $U= 1$ we have
\[
C_{T}^{R=1} = C_T \cap \{ R_0 < r < D \}, \qquad D \gg R_0
\]
\[
C_{T}^{U=1} = C_T \cap \{ -R_1 < t-r < 2\}.
\]
By $\tilde C_{T}^{R}$ and $\tilde C_{T}^{U}$ we denote enlargements of
these sets in both space and time on their respective scales.
We also
define 
\[
C_{T}^{<T/2} = \bigcup_{R < T/2} C_T^R,
\]
while $\tilde C_{T}^{<T/2}$ is a corresponding enlargement.
Finally, we will use the notation $C_{T}^{<T/2}(t_0) =
C_{T}^{<T/2}\cap \{t=t_0\}$ and similarly for $\tilde
C_{T}^{<T/2}(t_0)$.

The following Sobolev embeddings hold (see Lemma 3.8 from \cite{MTT}
for proof):
\begin{equation}
  \!  \| F_{\alpha\beta}\|_{L^\infty(C_T^{R})} \lesssim 
  \frac{1}{T^{\frac12} R^{\frac32}} \|F_{\alpha\beta}^{\leq 2}\|_{L^2( \tilde C_T^{R})} +  \frac{1}{T^{\frac12} R^{\frac12}}
  \|\nabla F_{\alpha\beta}^{\leq 2}\|_{L^2( \tilde C_T^{R})},
  \label{l2tolinf-r}\end{equation}
respectively
\begin{equation}
  \| F_{\alpha\beta}\|_{L^\infty(C_T^{U})} \lesssim \frac{1}{T^{\frac32} U^{\frac12}}  
  \|F_{\alpha\beta}^{\leq 2}\|_{L^2( \tilde C_T^{U})} +  \frac{U^{\frac12}}{T^{\frac32}}
  \|\nabla F_{\alpha\beta}^{\leq 2}\|_{L^2( \tilde C_T^{U})}.
  \label{l2tolinf-u}\end{equation}
 
We can now use \eqref{l2tolinf-r} and \eqref{l2tolinf-u} to improve the decay of the
derivative by a factor of $r^{-1}$ away from the cone, respectively by a factor of $\la t-r\ra^{-1}$ near the cone. We proved a similar type of result for the wave equation in \cite{MTT}, see Proposition 3.16, though the proof in that case was somewhat different.

 Let us first consider  the derivatives of $F$ in the region $C_T^R$. Clearly we have the bound 
$|e_{A,B} F_{\alpha\beta}^{\Lambda}| \lesssim \frac1r |F_{\alpha\beta}^{\Lambda+3}|$. For the time derivative, we use the fact that
\[
 \partial_t F_{\alpha\beta}^{\Lambda} = \frac1t S F_{\alpha\beta}^{\Lambda} - \frac{r}t \partial_r F_{\alpha\beta}^{\Lambda}. 
\] 
On the other hand, the Maxwell system gives us that
\[
 \partial_r F_{\alpha\beta}^{\Lambda} - \delta \partial_t F_{\tilde \alpha \tilde \beta}^{\Lambda} \in S^Z(1) (G_1^{\leq |\Lambda|+m}, G_2^{\leq |\Lambda|+m}) + S^Z(r^{-1}) F^{\leq |\Lambda|+m} 
\] 
for suitable $\tilde \alpha$ and $\tilde \beta$, where $\delta = \pm
1$ if $\alpha$ or $\beta$ equals $t$ or $r$  and $0$ otherwise. Combining the
last two relations we immediately get that
\begin{equation}\label{dtr-F}
 |\partial_{t, r} F_{\alpha\beta}^{\Lambda}| \lesssim \sum_{i=1}^2 |G_i^{\leq |\Lambda|+m}| + r^{-1}|F^{\leq |\Lambda|+m}|
\end{equation}

 After applying the Sobolev embeddings \eqref{l2tolinf-r} to $G_i$ we obtain that
\begin{equation} \label{l2toli:r}
R \| \nabla F_{\alpha\beta}^{\Lambda} \|_{L^\infty ( C_{T}^{R})} \lesssim \|F^{\leq |\Lambda|+m}\|_{L^\infty ( C_{T}^{R})} + \sum_{i=1}^2 T^{-\frac12}R^{\frac12}\|G_i^{\leq |\Lambda|+m}\|_{L^2 ( \tilde C_{T}^{R})}.
\end{equation}

 A similar argument applied in the region $C_T^U$ yields
\begin{equation}
    U \| \nabla F_{\alpha\beta}^{\Lambda} \|_{L^\infty ( C_{T}^{U})}
    \lesssim  \ \| F^{\leq |\Lambda| +m} \|_{L^\infty( 
      C_{T}^{U})} + T^{-\frac12} U^{\frac12} \sum_{i=1}^2\| G_i^{\leq
      |\Lambda|+m}\|_{L^2( \tilde C_{T}^{U})}.
  \label{l2toli:u}
  \end{equation}

\subsection{Improved bounds in the interior.} 
We will now obtain improved bounds for $F^{\Lambda}_{\alpha\beta}$ and
the gradient $\nabla F^{\Lambda}_{\alpha\beta}$ in the interior region
$C_{T}^{<T/2} $. We remark that the results in \cite{MTT}, namely
Proposition 3.14 and Proposition 3.15, which allow us to replace the
factor of $\la r\ra$ by a factor of $\la t\ra$ in the right hand side
of \eqref{frstest}, do not directly apply in the case of the Maxwell
system. Indeed, it is not clear why the stationary local energy decay
\eqref{sle} would hold for $F$ solving the system \eqref{compnts},
even if we assume that it holds for the corresponding scalar
equation. Moreover, we also need to deal with the presence of the
radial part of $F$, which decays at a different rate from the
nonradial part. Instead, we will be using the zero resolvent bound
\eqref{sle:vf} and Lemmas \ref{impl:levf} and \ref{impradpart} to
prove the following result which is similar to Proposition 3.15 in
\cite{MTT}.
  
\begin{proposition} \label{smallr} Assume that the solution to
  \eqref{Maxwelleqns} satisfies the zero resolvent bound
  \eqref{zero-res} for all $T\leq t_0\leq 2T$.  Then for  $m$
  large enough and any multiindex $\Lambda$ the following estimates
  hold:
 \begin{equation}\label{LEinsidedcy}
    \| \la r\ra^{-1}  F_{\alpha\beta}^{\Lambda}\|_{LE(C_T^{<\frac{T}2})} + \| \nabla F_{\alpha\beta}^{\Lambda}\|_{LE(C_T^{<\frac{T}2})} \lesssim M
  \end{equation}
and
\begin{equation}\label{insidedcy}
\begin{split}
    \!  \| F_{\alpha\beta}^{\Lambda}\|_{L^\infty(C_{T}^{<\frac{T}2})} +
\| \la r\ra \nabla F_{\alpha\beta}^{\Lambda}\|_{L^\infty(C_{T}^{<\frac{T}2})} \lesssim \tilde M
\end{split}
  \end{equation}
where 
 \begin{equation*}
\begin{split}
   M =  T^{-1} \|F^{\leq |\Lambda|+m}\|_{LE(\tilde C_{T}^{<T/2})} +
 \sum_{i=1}^2 \left( \|\la r\ra^{-1} G_i^{\leq |\Lambda|+m} \|_{LE^*(C_T)} + 
\| \la r\ra  \bar G_i^{\leq |\Lambda|+m} \|_{LE^*(C_T)}\right),    
\end{split}
  \end{equation*}
\begin{equation*}
\tilde M = T^{-\frac12} M + \sup_{R<T/2}\sum_{i=1}^2 T^{-\frac12} R^{\frac12}\|G_i^{\leq |\Lambda|+m}\|_{L^2 ( \tilde C_{T}^{R})}.
\end{equation*}  
\end{proposition}
 
\begin{proof}
  The main estimate here is the local energy bound
  \eqref{LEinsidedcy}.  Indeed, \eqref{insidedcy} follows from
  \eqref{LEinsidedcy} via the Klainerman-Sobolev type bounds
  \eqref{l2tolinf-r} and \eqref{l2toli:r} applied successively in all
  dyadic regions $R<\frac{T}2$. We remark that \eqref{LEinsidedcy} is
  the analogue of Proposition 3.14 in \cite{MTT}.

  To prove \eqref{LEinsidedcy}, we first note that, in view of
  Lemma~\ref{impradpart}, we can freely add to $M$ the corresponding
  bound for the radial part,
\[
M := M + T^{-1}\|\la r\ra^2 \bar F^{\leq |\Lambda|+m}\|_{LE(\tilde C_{T}^{<T/2})}.
\]
The next step is to localize the problem to $C_T$.
Let $\chi_T (t,r)$ be a nonnegative
smooth cutoff supported in $\tilde C_{T}^{<T/2}$ so that $\chi_T \equiv 1$ in
$C_{T}^{<T/2}$.  We replace $F$ with the tensor $\tilde F =
\chi_T F$.  We see that $\tilde F$ satisfies the system
  \[
  d\tilde F = \tilde G_1:= \chi_T G_1+ d \chi_T \wedge F, \qquad d*\tilde F = \tilde G_2:=\chi_T G_2+ d \chi_T \wedge * F.
 \]

  Clearly $\nabla \chi_T$ is supported in $\tilde C_{T}^{<T/2} \setminus C_{T}^{<T/2}$ and the cutoff can be chosen so that $|\nabla \chi_T| \lesssim T^{-1}$. We thus obtain that
\begin{multline*}
 \|\la r\ra^{-1} d \chi_T \wedge F\|_{LE^*(C_T)} + \| \la r\ra \overline
 {d \chi_T \wedge F}\|_{LE^*(C_T)} + \| \la r\ra \overline
 {d \chi_T \wedge *F}\|_{LE^*(C_T)} \\\lesssim T^{-1} \|F\|_{LE(\tilde
   C_{T}^{<T/2})}
 + T^{-1}\|\la r\ra^2 \overline {F}\|_{LE(\tilde C_{T}^{<T/2})}
\end{multline*}
where for the last term on the LHS we used that
\[
\overline{*F} \in S^Z(1)(\bar F) + S^Z(r^{-2})F
\]

After taking Lie derivatives, it is now easy to see that $\tilde G_i$ satisfies the inequality
\[
\|\la r\ra^{-1} \tilde G_i^{\leq \Lambda+m} \|_{LE^*(C_T)} + \| \la r\ra  \bar {\tilde{ G}}_i^{\leq \Lambda+m} \|_{LE^*(C_T)}
\lesssim M.
\]
Thus, from here on we assume that $F$ is (spatially) supported in $C_T^{<T/2}$ and drop the $\tilde F$ notation. 

The next step is to see that we can further replace $M$ by 
\[
M := M + T^{-1}(    \|\la r\ra \partial_r   F^{\leq |\Lambda|+m}\|_{LE(\tilde C_{T}^{<T/2})}   +  
 \|\la r\ra^3 \partial_r  \bar F^{\leq |\Lambda|+m}\|_{LE(\tilde C_{T}^{<T/2})}).
\]
Indeed, the first term on the right is estimated in terms of $M$ using the pointwise bound 
\eqref{dtr-F}, and for the second we use the relations \eqref{radparteqn1}-\eqref{radparteqn3}.

We now introduce the quantities
\[
\gamma^{|\Lambda|} = \sum_{i=1}^2   \|\la r\ra^{-1} G_i^{\leq |\Lambda|} \|_{LE^*(C_T)} ,
\qquad  
\bar \gamma^{|\Lambda|} =  \sum_{i=1}^2    \| \la r\ra  \bar G_i^{\leq |\Lambda|} \|_{LE^*(C_T)}
\]
respectively, with $h \in [0,1]$, 
\[
\phi^{|\Lambda|,h} =  T^h ( \| \la r\ra^{-h} F^{\leq |\Lambda|} \|_{LE(C_T)}
+ \| \la r\ra^{1-h} \nabla_x F^{\leq |\Lambda|} \|_{LE(C_T)}),
\]
\[
\bar \phi^{|\Lambda|,h} =  T^h ( \| \la r\ra^{2-h} \bar F^{\leq |\Lambda|} \|_{LE(C_T)}
+ \| \la r\ra^{3-h} \nabla_x \bar F^{\leq |\Lambda|} \|_{LE(C_T)}).
\]
With these notations, the bound to prove becomes
\begin{equation}\label{gamma-phi}
\phi^{|\Lambda|,1} + \bar \phi^{|\Lambda|,1} \lesssim \phi^{|\Lambda|+m,0} + \bar \phi^{|\Lambda|+m,0}
+ T \gamma^{|\Lambda|+m} + T \bar \gamma^{|\Lambda|+m}.
\end{equation}
Indeed, the time derivatives can be easily estimated afterwards by
  using either the Maxwell system or the scaling vector field $S$.

In order to use the bounds in Lemma~\ref{impl:levf} we need to convert 
the Maxwell system \eqref{Maxwelleqns} into the $d^0$ system \eqref{Maxwfixedt} and estimate the source terms $G_i^0$. We will show that
\begin{equation}\label{G_i^0} \begin{split}
& G^{0, \Lambda}_i \in S^Z(1) G^{\Lambda}_i + \frac1t dt \wedge S^Z(1) (F^{\leq |\Lambda|+m}, r\partial_r F^{\leq |\Lambda|+m}) \\
& \overline{G_i^0}^{\Lambda} \in S^Z(1) \overline{G_i}^{\Lambda} + \frac1t dt \wedge \bigl[S^Z(1) (\bar F^{\leq |\Lambda|+m}, r\partial_r \bar F^{\leq |\Lambda|+m}) + S^Z(r^{-2}) (F^{\leq |\Lambda|+m}, r\partial_r F^{\leq |\Lambda|+m}) \bigr]
\end{split} \end{equation}

 For this we use the scaling field $S$ as a proxy for $\partial_t$ to compute via \eqref{d0def}:
  \begin{equation}\label{G0} \begin{split}
   &   G^0_1 (t) =   G_1(t) - dt\wedge\L_{\partial_t}   F =    G_1(t) - \frac{1}{t}dt\wedge(\L_{S}   F - r \L_{\partial_r}   F - F_{r\phi}d\phi\wedge dr - F_{r\theta}d\theta\wedge dr) \\
      & \overline{G^0_1} (t) = \overline{G_1}(t) - dt\wedge\L_{\partial_t}  \bar F = \overline{G_1}(t) - \frac{1}{t}dt\wedge(\L_{S}  \bar F - r \L_{\partial_r} \bar  F)
 \end{split} \end{equation}
  
   We also need to take Lie derivatives in \eqref{G0}.
This is done using
  \[
  \L_X   G^0_1 (t) = \L_X   G_1(t) - dt\wedge\L_{\partial_t}
  F^{\le m}, \qquad \L_X   \overline{G^0_1} (t) = \L_X  \overline{G_1}(t) - dt\wedge\L_{\partial_t}
  \bar F^{\le m},
  \]

This  is immediate for $X\in \{\partial, \Omega\}$ and a simple computation
  for $X=S$. The desired result \eqref{G_i^0} for $G_1$ follows by induction.

 The proof of \eqref{G_i^0} for $G_2$ follows by applying the arguments above to $*F$ instead of $F$ and using the fact that
 \[
 \overline{*F} \in S^Z(1) \bar F + S^Z(r^{-2}) F
 \]
  
  We will now bound  the sources $G^0_i$ and their Lie derivatives  in $C_T$. We begin with the nonradial part, for which we have
  \begin{equation}\label{G0nonrad1/2}
\begin{split}
    \|\la r\ra^{-1} G^{0,\leq |\Lambda|}_i\|_{LE^*(C_T)} \lesssim & \
    \| \la r\ra^{-1} G^{\leq |\Lambda|}_i\|_{LE^*(C_T)} \\
&\qquad\qquad+ 
\frac1T \Bigl(\|\la r\ra^{-1}  F^{\leq
  |\Lambda|+m}\|_{LE^*(C_T)} +  \| \partial_r   F^{\leq |\Lambda|}\|_{LE^*(C_T)}\Bigr)
\\\lesssim  &\ 
  \gamma^{|\Lambda|}+  T^{-1}\phi^{|\Lambda|+m,\frac12}.
\end{split}  
\end{equation}
On the other hand for the radial part we have
  \begin{equation}\label{G0rad1/2}
\begin{split}
    \|\la r\ra \overline {G_i^0}^{\leq |\Lambda|}\|_{LE^*(C_T)}
    \lesssim & \ \|\la r\ra \bar G^{\leq
        |\Lambda|}_i\|_{LE^*(C_T)} 
\\&\ + \frac1T \Bigl(\|\la r\ra \bar F^{\leq |\Lambda|+m} \|_{LE^*(C_T)}
 + \|\la r\ra^2 \partial_r \bar F^{\leq |\Lambda|+m}\|_{LE^*(C_T)} 
\\ & \quad \quad + \|\la r\ra^{-1}  F^{\leq |\Lambda|+m}\|_{LE^*(C_T)}+ \| \partial_r  F^{\leq |\Lambda|+m}\|_{LE^*(C_T)}\Bigr)
\\
 \lesssim & \  \bar \gamma^{|\Lambda|} + T^{-1}\bar
 \phi^{|\Lambda|+m,\frac12}+  T^{-1}\phi^{|\Lambda|+m,\frac12}.
\end{split}  
\end{equation}
By interpolation we have bounds of the type 
\[
\phi^{|\Lambda|+m,\frac12} \lesssim (\phi^{|\Lambda|,1} \phi^{|\Lambda|+2m,0})^\frac12.
\]

Viewed in polar self-similar coordinates in $C^R_T$, these interpolation bounds are nothing but standard 
Sobolev bounds in a unit cube.  Using the interpolation estimates, from \eqref{G0nonrad1/2}
and \eqref{G0rad1/2} we have
\[ \begin{split}
 \|\la r\ra^{-1} G^{0\leq |\Lambda|}_i\|_{LE^*(C_T)} +  \|\la r\ra \overline{  G_i^0}^{\leq {|\Lambda|}}\|_{LE^*(C_T)} 
& \lesssim  \gamma^{|\Lambda|}+ \bar \gamma^{|\Lambda|} + 
  T^{-1} (\phi^{|\Lambda|,1} \phi^{|\Lambda|+2m,0})^\frac12 \\ &+
 T^{-1}(\bar \phi^{|\Lambda|,1} \bar \phi^{|\Lambda|+2m,0})^\frac12.
\end{split} \]

Now we can apply the zero resolvent bound \eqref{zero-res1}. Using
also the bounds \eqref{dtr-F} to estimate $\partial_r F^{\leq
  |\Lambda|}$, as well as \eqref{radparteqn1}-\eqref{radparteqn3} to
estimate $\partial_r \bar F^{\leq |\Lambda|}$, we obtain
\[
\phi^{|\Lambda|,1} + \bar \phi^{|\Lambda|,1} \lesssim  
  T\gamma^{|\Lambda|}+  T\bar \gamma^{|\Lambda|} +  (\phi^{|\Lambda|,1} \phi^{|\Lambda|+2m,0})^\frac12 +  (\bar \phi^{|\Lambda|,1} \bar \phi^{|\Lambda|+2m,0})^\frac12.
\]
Then the desired estimate \eqref{gamma-phi} follows by Cauchy-Schwarz.
\end{proof}

We will now prove \eqref{wkdecay} in the same way as in \cite{MTT}, by
a bootstrap procedure. The starting point is the pointwise bound
\eqref{frstest}. This can be improved by replacing the $r^{-1}$ factor
by a $t^{-1}$ factor, and complemented by a better bound for the
derivative near the cone. Indeed, by using \eqref{frstest} and
H\"older's inequality in \eqref{insidedcy} we obtain
\[
|F^{\Lambda}_{\alpha\beta}|\lesssim C_1 \frac{\log \la t-r\ra}{t \la
  t-r\ra^{\frac{1}{2}}},
\]
whereas using \eqref{frstest} in \eqref{l2toli:u} yields
\[
|\nabla F^{\Lambda}_{\alpha\beta}|\lesssim C_1 \frac{\log \la
  t-r\ra}{\la r \ra \la t-r\ra^{\frac{3}{2}}}.
\]
Here
\[
C_1 = \| F^{\leq |\Lambda|+m}\|_{LE} + \sum_{i=1}^2 \sup_{R,U} R^{\frac12} T^{\frac12}U^{\frac12} \|
G_i^{\leq |\Lambda| + m}\|_{L^2(C_T^{R,U})} + T \|r \bar G_i^{\leq |\Lambda| + m}\|_{LE^*(C_T)}.
\]
where, following \cite{MTT}, $C_T^{R,U}$ stands for either $C_T^R$
or $C_T^U$, with the convention that $R \approx T$ in $C_T^U$ and $U
\approx T$ in $ C_T^R$.

\subsection{Uniform pointwise bounds} 

We can now use the improved estimates in a bootstrap procedure similar
to the one in \cite{MTT}. The first step is to note that for a
solution to the Minkowski wave equation
\[
\Box w = f, \qquad w(0) = \partial_t w(0) = 0
\]
so that $f$ is supported in the forward cone $\{ t > |r|\}$
we can estimate
\begin{equation}\label{1D}
|w| \lesssim \frac1r \int_{D_{tr}} \int_{\S^2} \rho |f^{\leq m}| d\omega ds
d\rho 
\end{equation}
where $D_{tr}$ is the rectangle
\[
D_{tr}=\{ 0 \leq s - \rho \leq t-r, \quad t-r \leq s+\rho \leq t+r \}.
\]
We call this computation the one dimensional reduction; this is
fairly standard, and it is explained in detail in \cite{MTT}. By using
the above estimate in conjunction with \eqref{boxua} we improve our
estimate near the cone to
\[
|F^{\Lambda}_{\alpha\beta}| \lesssim C_2 \frac{\log\la t-r\ra}{\la r
  \ra \la t-r \ra}
\]
where
\[
C_2 = \| F^{\leq |\Lambda|+m}\|_{LE} + \sum_{i=1}^2 \sup_{R,U} TR^{\frac12}U^{\frac12} \|
G_i^{\leq |\Lambda| + m}\|_{L^2(C_T^{R,U})} + T^{\frac32}  \|
r \bar G_i^{\leq |\Lambda| + m}\|_{LE^*(C_T)}.
\]

Using this in \eqref{insidedcy} and \eqref{l2toli:u} improves the above estimate
to
\[
|F^{\Lambda}_{\alpha\beta}| \lesssim C_2 \frac{\log\la t-r\ra}{\la t
  \ra \la t-r \ra}, \qquad |\nabla F^{\Lambda}_{\alpha\beta}| \lesssim
C_2 \frac{\log\la t-r\ra}{\la r \ra \la t-r \ra^2}.
\]

One then again uses the pointwise estimates above in the one
dimensional reduction to improve the pointwise bounds near the cone,
followed by improving the bound inside through \eqref{insidedcy} and
derivative bounds near the cone through \eqref{l2toli:u}; see
\cite{MTT} for more details. After two more iterations,
\eqref{wkdecay} follows.

\subsection{Peeling estimates} 

Our starting point here is the bound \eqref{wkdecay}, which was
largely obtained using the wave equation solved by $F$. However, while
bounds of the form \eqref{wkdecay} are optimal for the scalar wave
equation, they are not optimal in the case of the Maxwell
tensor. Heuristically, this is due to the fact that at the
leading order the electromagnetic tensor $F$ is actually the
derivative of a potential, and derivatives of the solution decay better away from the
light cone even in the case of a scalar wave equation.
 
In order to improve our estimates, we first note that \eqref{wkdecay}
holds for the tensor components evaluated in the null frame. We shall
use the standard notation
\[
\phi_{-, A} = F_{uA}, \qquad \phi_0 = \frac12 (F_{uv}+iF_{AB}), \qquad
\phi_{+, A} = F_{vA}.
\]
 
We first note that $\partial_v F_{\alpha\beta}$ satisfies a better decay
bound than \eqref{wkdecay} near the cone. Indeed, since
\[
\partial_v = \frac1t S + \frac{t-r}t \partial_r
\]
by \eqref{wkdecay} we have 
\[
|\partial_v F^{\Lambda}| \lesssim \kappa_1 \frac{1}{\la r\ra \la t\ra
  \la t-r\ra^{2}}.
\]
It is also immediate that since $e_{A,B} \approx \frac1r \Omega$ we
 have
\[
|e_{A,B} F^{\Lambda}| \lesssim \kappa_1 \frac{1}{\la r\ra \la t\ra \la
  t-r\ra^{2}}.
\]
  
We would now like to improve the $\partial_u F$ term. By using the
Maxwell system, in particular
\[
\nabla^{\alpha }F_{\alpha u} \in S^Z(1) G_2, \qquad \nabla_{[u}F_{AB]} \in S^Z(1) G_1,
\qquad \nabla_{[u}F_{vA]} \in S^Z(1) G_1
\]
and \eqref{wkdecay} one obtains improved bounds for $\partial_u
\phi_0^{\Lambda}$ and $\partial_u F_{vA}^{\Lambda}$:
 \[
 |\partial_u \phi_0^{\Lambda}|, |\partial_u F_{vA}^{\Lambda}| \lesssim
 \kappa_1 \frac{1}{\la r\ra \la t\ra \la t-r\ra^{2}}.
 \]
 After integration on constant $v$ slices, one can improve the
 pointwise bounds near the cone to
 \begin{equation}\label{impcone}
   |\phi_0^{\Lambda}|, |F_{vA}^{\Lambda}| \lesssim \kappa_1 \frac{1}{ \la t\ra^2 \la t-r\ra}.
 \end{equation}

 In order to continue, let $\hat{\phi_0}:= (*F)_{AB} + i F_{AB}$.  We
 will derive the wave equation that $r\hat{\phi_0}$ satisfies.
 
\begin{lemma}\label{rphi}
 The middle component $\hat{\phi_0}$ satisfies
\begin{equation}\label{midcomp} \begin{split}
     \Box (r \hat{\phi_0}^{\Lambda}) \in & S^Z(r)(G_1^{\leq |\Lambda| +m},
     G_2^{\leq |\Lambda| +m}) + \partial(S^Z(r^{-1}) \hat{\phi_0}^{\leq
       |\Lambda| +m}) + \\ & S^Z(r^{-2})(\hat{\phi_0}^{\leq |\Lambda| +m})+
     \partial (S^Z (r^{-2}) F^{\leq |\Lambda| +m}) + S^Z (r^{-3})
     (F^{\leq |\Lambda| +m}). \end{split}
 \end{equation}
\end{lemma} 
\begin{proof}
 

 For spherically symmetric metrics of the form $h_{ab}dx^a dx^b + r^2d\omega^2$, the lemma has been proved in Section 4 of \cite{StTat}. We will adapt their proof for our more general case.
 
 Let us consider the metric (recall \eqref{normlr})
 \[
 \hat{g}= r^{-2}(1+g_\omega(r))^{-1} g = \tg + r^{-2} (1+g_\omega(r))^{-1}g_{sr}, \quad \tg := r^{-2}(1+g_{\omega}(r))^{-1}(-dt^2+dr^2)+ d\omega^2.
 \] 
 We note that in Cartesian coordinates the the metric coefficients of $\hat{g}$ and $\tilde g$ and their corresponding Christoffel symbols and curvature tensor satisfy the relations
\begin{equation}\label{confcurv}
\hat{g}_{\alpha\beta} \in S^Z(r^{-2}), \quad \hat{\Gamma}_{\alpha\beta}^{\gamma}\in S^Z(r^{-1}), \quad \hat{R}_{\alpha\beta\gamma}^{\delta} \in S^Z(r^{-2})
\end{equation}
\begin{equation}\label{confdiff}
\hat{g}_{\alpha\beta}-\tg_{\alpha\beta} \in S^Z(r^{-4}), \quad \hat{\Gamma}_{\alpha\beta}^{\gamma} - \tilde \Gamma_{\alpha\beta}^{\gamma} \in S^Z(r^{-3}), \quad \hat{R}_{\alpha\beta\gamma}^{\delta} - 
\tilde R_{\alpha\beta\gamma}^{\delta} \in S^Z(r^{-4}).
\end{equation}
 We will also use the spherical coordinates $(t, r, \phi,\theta)$ so
 that locally 
$\{e_A, e_B\} = \{\frac1r \partial_{\phi}, \frac1{r\sin\theta}\partial_{\theta}\}$. 
  Since the Maxwell system is conformally invariant, we have due to \eqref{wave_form1} that
\[
\Box_{\hat{g}} F_{CD} - \hat{R}_{C}^{\ \ \gamma}F_{\gamma D}
  - \hat{R}_{D}^{\ \ \gamma}F_{C\gamma}
  + \hat{R}_{CD}^{\ \ \ \gamma\delta}F_{\gamma\delta} \in S^Z(r^4) (G_1^{\leq 1}, G_2^{\leq 1}), \qquad C, D\in\{\phi, \theta\}.
\] 
 Moreover the metric $\tg$ is of the form described in Section 4 of \cite{StTat}, so it is easy to check that
\[
-\tilde R_{C}^{\ \ \gamma}F_{\gamma D}
  - \tilde R_{D}^{\ \ \gamma}F_{C\gamma}
  + \tilde R_{CD}^{\ \ \ \gamma\delta}F_{\gamma\delta} = 0,
\] 
which implies, due to \eqref{confdiff} that
\[
(\hat{R}_{C}^{\ \ \gamma}-\tilde R_{C}^{\ \ \gamma}) F_{\gamma D}, \quad (\hat{R}_{C}^{\ \ \gamma}-\tilde R_{D}^{\ \ \gamma}) F_{C\gamma},  \quad (\hat{R}_{CD}^{\ \ \ \gamma\delta} - \tilde R_{CD}^{\ \ \ \gamma\delta}) F_{\gamma\delta} \in S^Z(1) F.
\] 
 We thus obtain 
\begin{equation}\label{FCD1}
\Box_{\hat{g}} F_{CD} \in S^Z(r^4) (G_1^{\leq 1}, G_2^{\leq 1}) + S^Z(1) F.
\end{equation}

 We can now trace with respect to the the volume form $\epsilon^{CD}$ of $\S^2$. We obtain, with $\alpha$ signifying Cartesian coordinates:
\[
\Box_{\hat{g}} (\epsilon^{CD} F_{CD}) - \epsilon^{CD} \Box_{\hat{g}} F_{CD} = 2\hat{\nabla}_{\alpha} ((\hat{\nabla}^{\alpha}\epsilon^{CD}) F_{CD}) - (\hat{\nabla}_{\alpha}\hat{\nabla}^{\alpha} \epsilon^{CD}) F_{CD}.
\] 

 Since for the spherically symmetric metric $\tilde g$ we know that $\tilde \nabla \epsilon = 0$, we obtain that
\[
\hat{\nabla}^{\alpha}\epsilon^{CD} \in S^Z(r^{-1})
\] 
and thus, taking into account that $F_{CD} \in S^Z(r^2) F$,
\begin{equation}\label{FCD2}
\Box_{\hat{g}} (\epsilon^{CD} F_{CD}) - \epsilon^{CD} \Box_{\hat{g}} F_{CD} \in \partial (S^Z(r) F) + S^Z(1) F. 
\end{equation}
 Finally, \eqref{FCD1}, \eqref{FCD2} and the fact that $\epsilon^{CD} F_{CD} = r^2 F_{AB}$ imply
\begin{equation}\label{FAB}
\Box_{\hat{g}} [r^2 F_{AB}] \in S^Z(r^4) (G_1^{\leq 1}, G_2^{\leq 1}) + S^Z(1) F + \partial (S^Z(r) F).
\end{equation}

On the other hand, for any function $\psi$ we have
\begin{equation}\label{confwave}
\Box_g \psi -\frac16 R \psi = r^{-3} \Bigl(\Box_{\hat{g}} (r\psi) -\frac16 \hat{R} (r\psi)\Bigr).
\end{equation}
In particular for $\psi= r F_{AB}$ we obtain, using \eqref{FAB},
\[
\Box_g [r F_{AB}] \in S^Z(r) (G_1^{\leq 1}, G_2^{\leq 1}) + S^Z(r^{-2})(F_{AB})+ \partial (S^Z (r^{-2}) F) + S^Z (r^{-3}) (F).
\]
One obtains a similar equation for $(*F)_{AB}$. We thus get
\[
\Box_g [r \hat{\phi_0}] \in S^Z(r) (G_1^{\leq 1}, G_2^{\leq 1}) + S^Z(r^{-2})(\hat{\phi_0})+ \partial (S^Z (r^{-2}) F) + S^Z (r^{-3}) (F).
\]
%
%
%
%
 
 We can now replace $\Box_g$ above by the operator $P=\Box + Q$ from \eqref{P}, and  
 the conclusion follows after commuting with elements of $Z$. 
\end{proof} 
 
 We can now apply the following lemma, proved in \cite{MTT} (Lemma 3.20) for $m=-2, 1$:
 
\begin{lemma}\label{imprvddecay}
  Consider a smooth function $f$ supported in $ \{\frac{t}2 \leq r
  \leq t \}$ so that
  \begin{equation}
    |f|+|Sf|+|\Omega f|  + \la t-r \ra |\partial_r f|  \lesssim \frac{\log^m\la t-r \ra}{t^3 \la t-r \ra}, \quad m\in\mathbb Z.
    \label{fall}\end{equation}
  and $h$ supported in $ \{0 < r_e \leq r \leq t \}$ so that
  \begin{equation}\label{hest}
  |h| \lesssim \frac{\log^m\la t-r \ra}{t r^{3} \la t-r \ra}.
  \end{equation}
  Then the forward solution $w$ to
  \[
  \Box w = \partial_t f + h
  \]
  satisfies the bound
  \begin{equation}
    | w| \lesssim \frac{\log^{m+2}\la t-r \ra}{t \la t-r \ra^2}.
    \label{fallu}\end{equation}
  \label{l:maindr}
\end{lemma}

\begin{proof}

We write $w=w_1+w_2$, where 
\[
\Box w_1 = \partial_t f, \qquad \Box w_2 = h.
\]

 Let us first bound $w_2$. We use the one-dimensional reduction \eqref{1D} and decompose $D_{tr}$ dyadically as 
 \[
D_{tr} = \bigcup_{R \leq t}  D_{tr}^R, \quad D_{tr}^R = D_{tr}\cap \{R<\rho<2R\}.
\]
For $R < (t-r)/8$, we have that $\rho\approx R$ and $\la s-\rho\ra
\approx \la t-r\ra$ in $D_{tr}^R $.  Thus \eqref{hest} implies that
\[
\int_{D_{tr}^R} \int_{\S^2} \rho |h^{\leq n}| d\omega \,ds\,d\rho
\lesssim \frac{\log^{m}\la t-r \ra}{\la t-r\ra^2}. 
\]
 On the other hand, for $R>(t-r)/8$ we obtain by \eqref{hest} (here $u=s-\rho$):
\[
\int_{D_{tr}^R} \int_{\S^2} \rho |h^{\leq n}| d\omega\,ds\,d\rho \lesssim R^{-2} \int_{0}^{t-r} \frac{\log^m\la u \ra}{\la u\ra} du \lesssim R^{-2} \log^{m+2}\la t-r\ra.
\] 
Dyadically summing in $R$ shows that 
\[|w_2|\lesssim \frac{\log^{m+2}\la t-r\ra}{r\la t-r\ra^2},\]
and using this bound in \eqref{insidedcy} yields \eqref{fallu}.

On the other hand, we can write $w_1 = \partial_t v$, where $v$ is the forward solution to $\Box v = f$. We first note that \eqref{fall} also implies that
\[
|(t \partial_i + x_i \partial_t) f| \lesssim  \frac{\log^m\la t-r \ra}{t^n \la t-r \ra}.
\]
Via the one dimensional reduction as above applied to $v$, $\nabla v$ 
$\Omega v$, $S v$ and $(t \partial_i
+ x_i \partial_t) v$ we obtain 
\[
|v|+|\nabla v| +|Sv|+|\Omega v|  + \sum_i  | (t \partial_i
+ x_i \partial_t) v|  \lesssim \frac{\log^{m+2}\la t-r\ra}{t \la t-r \ra} .
\]
The above left hand side dominates $\la t-r\ra \partial_t v$;
therefore the proof of the lemma is complete.
\end{proof}

In order to use the lemma, we recall that 
\[
\partial = \partial_t + S^Z(\frac 1r)(S, \Omega ).
\]
Let $\chi$ be a smooth cutoff selecting the region $r \leq t/2$ . We can rewrite \eqref{midcomp} as
\[
\Box (r \hat{\phi_0}^{\Lambda}) = S^Z(r)(G_1^{\leq |\Lambda| +m}, G_2^{\leq |\Lambda| +m}) + \partial_t f + h
\]
where
\[
f = (1-\chi)\Bigl(S^Z(r^{-1}) \hat{\phi_0}^{\leq |\Lambda| +m} + S^Z (r^{-2}) F^{\leq |\Lambda| +m}\Bigr)
\]
\[
h = \partial_t \Bigl(\chi(S^Z(r^{-1}) \hat{\phi_0}^{\leq
       |\Lambda| +m} + S^Z (r^{-2}) F^{\leq |\Lambda| +m})\Bigr) + S^Z(r^{-2})(\hat{\phi_0}^{\leq |\Lambda| +m}) + S^Z (r^{-3})(F^{\leq |\Lambda| +m})
\]
%
 Due to \eqref{midcomp}, \eqref{impcone}, \eqref{wkdecay}, and \eqref{l2toli:r} we see that
\[
|f|+|Sf|+|\Omega f| + \la t-r \ra |\partial_r f|  \lesssim \frac{1}{t^3 \la t-r \ra}, \quad |h| \lesssim \frac{1}{t r^{3} \la t-r \ra}.
\] 
We can now use Lemma~\ref{imprvddecay} (with $m=0$),  to improve the
bounds on $\hat{\phi_0}$  to
\begin{equation}\label{impcone2}
  |\hat{\phi_0}^{\Lambda}| \lesssim \kappa \frac{\log^2 \la t-r \ra}{r \la t \ra \la t-r \ra^2}.
\end{equation} 
 We can use the new bound \eqref{impcone2} to improve the decay of $f$ and $h$ to
\[
|f|+|Sf|+|\Omega f| + \la t-r \ra |\partial_r f|  \lesssim \frac{1}{t^3 \la t-r \ra\log^2\la t-r\ra}, \quad |h| \lesssim \frac{1}{t r^{3} \la t-r \ra\log^2\la t-r\ra}.
\]  
By applying Lemma~\ref{imprvddecay} (with $m=-2$) we can remove the logarithm in \eqref{impcone2}, and obtain the desired estimate for
$\phi_0$ near the cone:
\begin{equation}\label{impcone3}
  |\phi_0^{\Lambda}| \lesssim \kappa \frac{1}{r \la t \ra \la t-r \ra^2}.
\end{equation} 
 We will later improve in the region $r\leq t/2$ by replacing the $r$ factor in the denominator by a $t$ factor.

Note that \eqref{impcone3} implies, in conjunction with
\eqref{l2toli:u}, that in the region $\{r>t/2\}$ we have
\begin{equation}\label{derivest}
  |\nabla \phi_0^{\Lambda}| \lesssim \kappa \frac{1}{r^2  \la t-r \ra^3}.
\end{equation}
Since $\phi_{-}$ is a one-form on the sphere, Hodge theory tells us
that
\[
\|\phi_{-}\|_{L^2(\S^2)} \lesssim \|\ang \phi_{-}\|_{L^2(\S^2)}.
\]
On the other hand, part of the Maxwell system gives that
\begin{equation}\label{Maxpart}
\nabla^{\alpha }F_{\alpha u} \in S^Z(1) G_2, \qquad \nabla_{[u}F_{AB]} \in S^Z(1) G_1
\end{equation}
 which in turn implies that
\[
\|\ang \phi_{-}\|_{L^{\infty}(|x|=R)} \lesssim \|\partial_u
\phi_0\|_{L^{\infty}(|x|=R)} + \frac1R \|F\|_{L^{\infty}(|x|=R)} +
\sum_{i=1}^2 \|G_i\|_{L^{\infty}(|x|=R)}.
\]
After taking derivatives in \eqref{Maxpart} and applying Sobolev embeddings on the
sphere of radius $r$, we obtain the desired bound for $F_{uA}$ near the cone:
\begin{equation}\label{impcone4}
|F_{uA}^{\Lambda}| \lesssim \kappa \frac{1}{\la r\ra \la t-r\ra^{3}}.
\end{equation}

To complete the proof of the peeling estimates near the cone, we note that the Maxwell system, in particular the equations
\[
\nabla_{[u} F_{vA]} = (G_1)_{uvA}, \qquad \nabla^{\alpha} F_{\alpha A} =-
*(G_2)_{A}
\]
combined with the previous bounds for $\phi_{-}$ and $\phi_0$ imply
that
\[
|\partial_u F_{vA}^{\Lambda}| \lesssim \kappa \frac{1}{r^2 \la t \ra
  \la t-r \ra^2}.
\]
After integration on constant $v$ slices, we obtain the desired bound for
$F_{vA}$ near the cone:
\begin{equation}\label{impcone5}
|F_{vA}^{\Lambda}| \lesssim \kappa \frac{1}{r \la t \ra \la t-r \ra^2}.
\end{equation}

Finally, we note that \eqref{impcone3}, \eqref{impcone4} and \eqref{impcone5} imply that $|F^{\leq n}| \lesssim \kappa \frac{1}{\la r \ra 
\la  t-r \ra^3}$. Taking into account \eqref{insidedcy}, we can replace the $r$ factor in the denominator with a
$t$ factor and conclude the proof. One can also easily see that \eqref{pointwiseest1} follows from the proof of Lemma~\ref{impradpart} by using the decay of $F$ just proved.

\end{document}